\documentclass[a4paper,11pt]{amsart}%  {report} %
\usepackage{mathrsfs}
\usepackage{amsfonts}
\usepackage[all]{xy}
\usepackage{amsmath}
\usepackage{arcs}
\usepackage{amscd}
\usepackage{CJK}
\usepackage[pdftex]{graphicx,xcolor}
\usepackage{pdfpages} %% for insert pdf file
\usepackage{bm}
\usepackage{amssymb}
\usepackage{amscd,latexsym,amsthm,amsfonts,amssymb,amsmath,amsxtra}
\newtheorem{prop}{Proposition}%[section]
\newtheorem{thm}{Theorem}%[section]
\newtheorem{lemma}{Lemma}%[section]
\author{Zhangjie Wang}

\newcommand{\En}{E^{(n)}}
\newcommand{\cEn}{\cE^{(n)}}

\newcommand{\Sel}{{{\mathrm{Sel}}}}

\newcommand{\Leg}[2]{{\brlr{\frac{#1}{#2}}}}
\newcommand{\ALeg}[2]{{\Sqlr{\frac{#1}{#2}}}}

\newcommand{\arr}{\ar@}
\newcommand{\arro}{\ar@/}

\newcommand{\Kc}{{K^\times}}

\newcommand{\tor}{{\mathrm{tor}}}

\newcommand{\oset}{\overset}

\newcommand{\Ma}[4]{{
\left(
  \begin{array}{cc}
    #1 & #2 \\
    #3 & #4 \\
  \end{array}
\right)
}}

\newcommand{\Ba}[2]{{\langle #1, #2\rangle}}

\newcommand{\ABlr}[1]{{\left|#1 \right|}}
\newcommand{\ABbig}[1]{{\big| #1\big|}}

\newcommand{\brlr}[1]{{\left( #1\right)}}
\newcommand{\brbig}[1]{{\big(#1 \big)}}
\newcommand{\brBig}[1]{{\Big(#1 \Big)}}
\newcommand{\brbigg}[1]{{\bigg( #1\bigg)}}

\newcommand{\Brlr}[1]{{\left\{ #1\right\}}}
\newcommand{\Brbig}[1]{{\big\{ #1\big\}}}
\newcommand{\BrBig}[1]{{\Big\{ #1\Big\}}}

\newcommand{\Sqlr}[1]{{\left[#1 \right]}}
\newcommand{\Sqbig}[1]{{\big[ #1\big]}}
\newcommand{\SqBig}[1]{{\Big[#1 \Big]}}

\font\cyr=wncyr10 scaled \magstep 1 % 为了添加Tate-Shafarevich gp
\newcommand{\Sha}{\mbox{\cyr X}} % Tate-Shafarevich group

\newcommand{\diag}{{\mathrm{diag}}}

 \renewcommand{\Im}{{\mathrm{Im}}}

 \newcommand{\rank}{{\mathrm{rank}}}

\renewcommand{\Re}{{\mathrm{Re}}}

\renewcommand{\mod}{\mathrm{mod}}

\newcommand{\half}{\frac{1}{2}}

\newcommand{\barw}{{\bar {w}}} 
 \newcommand{\barz}{{\bar {z}}}

\newcommand{\bA}{{\mathbb {A}}} 
\newcommand{\bC}{{\mathbb {C}}} 
 \newcommand{\bF}{{\mathbb {F}}}

 \newcommand{\bN}{{\mathbb {N}}}
 
\newcommand{\bQ}{{\mathbb {Q}}} \newcommand{\bR}{{\mathbb {R}}}

 \newcommand{\bZ}{{\mathbb {Z}}}

\newcommand{\cA}{{\mathcal {A}}} 
 \newcommand{\cD}{{\mathcal {D}}}
\newcommand{\cE}{{\mathcal {E}}} 
\newcommand{\cG}{{\mathcal {G}}}

\newcommand{\cM}{{\mathcal {M}}} 
\newcommand{\cO}{{\mathcal {O}}} \newcommand{\cP}{{\mathcal {P}}}
 \newcommand{\cR}{{\mathcal {R}}}

\newcommand{\sA}{{\mathscr{A}}} \newcommand{\sB}{{\mathscr{B}}}

\newcommand{\sI}{{\mathscr{I}}}

 \newcommand{\sP}{{\mathscr{P}}}
\newcommand{\sQ}{{\mathscr{Q}}} 
\newcommand{\sS}{{\mathscr{S}}} \newcommand{\sT}{{\mathscr{T}}}

\newcommand{\fa}{{\mathfrak {a}}} \newcommand{\fb}{{\mathfrak {b}}}
\newcommand{\fc}{{\mathfrak {c}}}

 \newcommand{\fp}{{\mathfrak {p}}}

\newcommand{\Obig}[1]{O\brbig{#1}}

\newcommand{\oBig}[1]{o\brBig{#1}}

\newcommand{\Li}{{\mathrm{Li}}}
\newcommand{\bQc}{{\bQ^\times}}

\numberwithin{equation}{section}
\title{Non-trivial Shafarevich-Tate Groups of Elliptic Curves}
\begin{document}
\begin{CJK}{UTF8}{gbsn} %{GBK}{kai}
\maketitle
% \tableofcontents
\begin{abstract}
We characterize quadratic twists of $y^2=x(x-a^2)(x+b^2)$ with Mordell-Weil groups
and $2$-primary part of Shafarevich-Tate groups being isomorphic to $\brbig{\bZ/2\bZ}^2$ under certain conditions. We also obtain the distribution result of these elliptic curves.\\



\end{abstract}

\section{Introduction}

In our previous paper \cite{wang2016-2}, we use Cassels  pairing to characterize congruent elliptic curves $y^2=x^3-n^2x$  with Mordell-Weil ranks zero and $2$-primary parts of Shafarevich-Tate groups  being isomorphic to $\brbig{\bZ/2\bZ}^2$ provided that all prime divisors of $n$ are congruent to $1$ modulo $4$. On the other hand side, we use the independence property of residue symbols to obtain corresponding distribution results in \cite{wang2016-3}. These tools play an important role in the proof of the breakthrough of Smith \cite{Smith-Selmer} on the distribution of $2$-Selmer groups.  The goal of this paper is to   generalize these methods to quadratic twist family of elliptic curves with full rational  $2$-torsion points.

Let $a$ and $b$ be coprime  integers. Denote by  $E=E_{a,b} $ the elliptic curve
\[E: y^2=x(x-a^2)(x+b^2).\]
Then the quadratic twist family of $E$ consists of the elliptic curves
\begin{equation}\label{eq-En}
E^{(n)}: y^2=x(x-a^2n)(x+b^2n),
\end{equation}
where $n$ runs over all non-zero square-free integers. Note that if $a=b=1$, these are  the congruent elliptic curves. To state our main theorem, we introduce some notation. For  a positive square-free integer $m$ and a positive integer $k$, the $2^k$-rank  $h_{2^k}(m)$ of the ideal class group $\cA=\cA_m$ of  $\bQ(\sqrt{-m})$ is defined to be \[\dim_{\bF_2} 2^{k-1}\cA/2^k\cA.\] Here   the group operation of $\cA$ is written additively.

\begin{thm}\label{thm-ab odd square-p +-1 mod 8-introd}
Let  $(a,b,c)$ be any positive primitive  integer solution to $a^2+b^2=2c^2$ such that the  dimension of the $2$-Selmer group of $E$  is two.  Denote by $n$   a positive square-free integer   such that  all prime factors of $n$ are congruent to   $\pm1$ modulo $8$ and quadratic residues modulo any prime divisor $p$ of $abc$. If $n\equiv1\pmod8$, then the following are equivalent:
\begin{enumerate}
\item[(1)] $E^{(n)} (\bQ)$ and $\Sha(E^{(n)}/\bQ)[2^\infty]$ are isomorphic to $ \brlr{\bZ/2\bZ}^2$,
\item[(2)]  $h_4(n)=1$  {\rm and}  $h_8(n)=0.$
\end{enumerate}
\end{thm}

From elementary number theory,  if $(a,b,c)$ is any positive primitive   integer
solution to $a^2+b^2=2c^2$, then $a, b$ and $c$ are the absolute values of
$4k^2-4k-1, 4k^2+4k-1$ and $4k^2+1$ respectively, where $k$ is an integer.
For example, if $k=2, 3, 6, 7, 9, 10, 11,\cdots$, then the dimensions of the
$2$-Selmer groups of the corresponding $E_{a,b}$ are two. Among all positive
integers no larger than $50$, there are $19$ such $E_{a,b}$.

To state another theorem, we have to use Gauss genus theory (refer to \S 2.4 of this paper or \S 3 of \cite{wang2016-2} for more details).
If $m$ is odd and  $h_4(m)=1$, then   there are exactly two divisors $d_1$ and $d_2$ of $2m$ which correspond to the non-trivial element of $2\cA_m\cap \cA_m[2]$, where $\cA_m[2]$ denotes the ideal classes with trivial squares. Furthermore, the product of the odd parts of $d_1$ and $d_2$ is $m$.

\begin{thm}\label{thm-ab odd square-p 1 mod 4-introd}
Let  $(a,b,c)$ be any positive primitive  integer solution to $a^2+b^2=2c^2$ such that the  dimension of the $2$-Selmer group of $E$  is two. Assume that    $n$ is a positive square-free integer   such that all prime factors of $n$ are quadratic residues modulo $4p$ with $p$ any prime divisor of  $abc$. If  $n$ is congruent to $1$ modulo $8$, then the following are equivalent:
\begin{enumerate}
\item[(1)] $E^{(n)} (\bQ)$ and $\Sha(E^{(n)}/\bQ)[2^\infty]$ are isomorphic to $ \brlr{\bZ/2\bZ}^2$,
\item[(2)]    $h_4(n)=1$  {\rm and}  $h_8(n)\equiv\frac{d-1}4 \pmod2.$
\end{enumerate}
Here $d$ denotes the odd part of $d_0$ which corresponds to the non-trivial element of $2\cA\cap \cA[2]$.
%, and $\delta$ is $1$ if $\rank A_n=k-1$ and $0$ otherwise.
\end{thm}
Now we explain the idea of the proof of Theorem \ref{thm-ab odd square-p +-1 mod 8-introd} and \ref{thm-ab odd square-p 1 mod 4-introd}. Via the    exact sequence
\begin{equation}\label{sel-2 condition}
0\to  E^{(n)}(\bQ)/2 E^{(n)}(\bQ)\to \Sel_2( E^{(n)})\to \Sha( E^{(n)}/\bQ)[2]\to 0,
\end{equation}
we derive that (1)  implies $s_2(n)=2$. Here $s_2(n)$ is the pure $2$-Selmer rank
$$\dim_{\bF_2}\Sel_2( E^{(n)})/ E^{(n)}(\bQ)[2].$$ By Gauss genus theory, $h_4(n)$ is closely related to the R\'edei matrix $R_n$; we have  parallel results  between $s_2(n)$ and the generalized  Monsky matrix $M_n$ by Proposition \ref{prop-sel2-rep}.  Then we get  that  $s_2(n)=2$ if and only if $h_4(n)=1$.  Cassels \cite{Cassels1998CasPairing} introduced a skew-symmetric pairing on the pure $2$-Selmer group $\Sel_2( E^{(n)})/ E^{(n)}(\bQ)[2]$. We can  show that (1) is equivalent to  the non-degeneracy of the Cassels pairing provided that $h_4(n)=1$ . According to Cassels pairing and Gauss genus theory, the non-degeneracy of the Cassels pairing under this condition is equivalent to (2).

To give the distribution result on the elliptic curves in Theorem \ref{thm-ab odd square-p 1 mod 4-introd}, we first introduce some notation.
Let $a, b$ and $c $ be coprime positive integers such that $a^2+b^2=2c^2$ and the
dimension of the $2$-Selmer group of $ E$  is two.
Let $k$ be a fixed positive integer. We denote by $\sQ_k(x)$ the set of positive square-free integers $n=p_1\cdots p_k\le x$ satisfying
\begin{itemize}
\item $n$ is congruent to $1$ modulo $8$, and
\item  all $p_l$ are quadratic residues modulo $4p$ with $p$ any prime divisor of $abc$.
\end{itemize}We define $\sP_k(x)$   to be all $n\in\sQ_k(x)$ such that
\begin{equation}\label{Pk(x)}
\rank_\bZ \En(\bQ)=0\quad {\rm and}\quad \Sha(\En/\bQ)[2^\infty]\simeq \brbig{\bZ/2\bZ}^2.
\end{equation} Denote by $C_k(x)$ the set of positive square-free integers $n\le x$ with exactly $k$ prime factors. Then the independence property  of Legendre symbols of Rhoades \cite{rhoades20092} implies
\[\#C_k(x)\sim \frac1{(k-1)!}\frac{x(\log \log x)^{k-1}}{\log x}.\]
Here the symbol ''$\sim$'' and many other symbols ''$\ll, \Obig{\cdot}, o\brbig{\cdot}, \Li(x) $'' are   standard notation in analytic number theory, it can be found in many references such as Iwaniec-Kowalski \cite{ireland1982classical}.
Let $\lfloor \frac{k}{2}\rfloor$ be the maximal integer no larger than  $k/2$. We define    $\Brbig{u_k: k\in\bN}$ to be the decreasing sequence $\BrBig{ \displaystyle{ \prod_{i=1}^{\lfloor \frac{k}{2}\rfloor } (1-2^{1-2i})}: k\in\bN } $ with limit  $u\approx0.419$.

\begin{thm}\label{thm-dist}
Let $a, b$ and $c $ be coprime positive integers such that $a^2+b^2=2c^2$ and the
dimension of the $2$-Selmer group of $E$  is two.
Then for any positive integer $k$,
\[\#\sP_k(x)\sim 2^{-kk'-k-2}\brBig{u_k+ (2^{-1}-2^{-k})u_{k-1}} \cdot \#C_k(x). \]
Here $k'$ is the number of different prime factors of $abc$.
\end{thm}
The key ingredient of the proof of Theorem \ref{thm-dist} is the independence property of residue symbols (Theorem \ref{thm-independence}), which reduces   counting $\#\sP_k(x)$ to counting certain symmetric $k\times k$ matrices over $\bF_2$.

In the end of this introduction, we introduce the arrangement of this paper. In Section 2, we introduce some preliminary results and several residue symbols. Section 3 is focused on the matrix representation of $2$-Selmer group. We prove that $\En_\tor(\bQ)$ is isomorphic to $\brbig{\bZ/2\bZ}^2$ is Section 4.  We devote Section 5 to prove Theorem \ref{thm-ab odd square-p +-1 mod 8-introd} and \ref{thm-ab odd square-p 1 mod 4-introd}. We use the method of Cremona-Odoni \cite{cremona1989some} to prove the independence property of residue symbols (Theorem \ref{thm-independence}) in Section 6. In the last section, we prove the distribution result (Theorem \ref{thm-dist}).

\section{Preliminary section}
\subsection{Identification of $2$-Selmer Group}\quad

Let $E/\bQ$ be an elliptic curve with full rational $2$-torsion points defined by
\begin{equation}\label{eq:E-equation-general}
E: y^2=(x-a_1)(x-a_2)(x-a_3).
\end{equation}
Then we can identify (see Cassels \cite{Cassels1998CasPairing}) the $2$-Selmer group $\Sel_2(E)$ of $E$ with
\[\BrBig{ \Lambda=(d_1,d_2,d_3)\in (\bQc/\bQ^{\times2})^{3 } \;\Big|\; d_1d_2d_3\in\bQ^{\times2},\; D_\Lambda(\bA)\not=\emptyset}.\]
Here $\bA$ is the adele ring of $\bQ$ and $D_\Lambda$ is a genus one curve defined by
\begin{equation}\label{D-Lambda-general}
\left\{ \begin{aligned}
H_1:& (a_2-a_3)t^2+d_2u_2^2-d_3u_3^2=0, \\
H_2:& (a_3-a_1)t^2+d_3u_3^2-d_1u_1^2=0, \\
H_3:& (a_1-a_2)t^2+d_1u_1^2-d_2u_2^2=0.
\end{aligned} \right.
\end{equation}
Moreover, $E(\bQ)/2E(\bQ)$ can be embedded into $\Sel_2(E)$. If $(x,y)\not\in E(\bQ)[2]$ is a rational point on $E$, the embedding  is given by $(x,y)\mapsto (x-a_1,x-a_2,x-a_3)$. The $2$-torsion point  $(a_1,0)$ corresponds to
$$\brbig{ (a_1-a_2)(a_1-a_3), a_1-a_2, a_1-a_3}.$$ Similar correspondences are defined   for the $2$-torsion points $(a_2,0)$ and $(a_3,0)$.

\subsection{Cassels Pairing}\quad

For a general elliptic curve $E/\bQ$,
Cassels \cite{Cassels1998CasPairing} defined a skew-symmetric bilinear pairing
$\Ba{ \cdot}{\cdot }$ on the pure $2$-Selmer group $\Sel_2'(E)$, which is an $\bF_2$-vector space  defined by $\Sel_2(E)/E(\bQ)[2]$. We assume that $E$ is defined by the equation (\ref{eq:E-equation-general}), and we use the identification of $\Sel_2(E) $ in \S2.1.   Let $\Lambda=(d_1,d_2,d_3)$ be any element of $\Sel_2(E)$ and $D_\Lambda$ the corresponding   genus one curve  associated to $\Lambda$. Since $H_i$ is locally solvable everywhere,   there is a $Q_i\in H_i(\bQ)$ by Hasse-Minkowski principle.  We define $L_i$ to be a linear form in three variables (all of $t, u_1, u_2, u_3$   except $u_i$) such that $L_i=0$ defines the  tangent plane of $H_i$ at $Q_i$. Then we   call $L_i$ the tangent linear form of $H_i$ at $Q_i$. Moreover, we consider it as a linear form in $u_1, u_2, u_3$ and $t$ with the coefficient of $u_i$ being zero. As $D_\Lambda$ is locally solvable everywhere, there are enough points on $D_\Lambda(\bA)$ such that we may choose   $P=(P_p)\in D_\Lambda(\bA)$ with all $\prod_{i=1}^3L_i(P_p)$ non-vanishing.  Given any $ \Lambda'=(d_1',d_2',d_3')\in \Sel_2(E)$,   the local Cassels pairing $\Ba\Lambda{\Lambda'}_p$ is defined  to be
\[\prod_{i=1}^3 \brBig{L_i(P_p), d_i'}_p.\]
Here $p$ is any rational prime or infinity, and  $\brbig{\cdot,\cdot}_p$ denotes the Hilbert symbol at $\bQ_p$ ($\bQ_\infty=\bR$ if $p=\infty$). Then the Cassels pairing $\Ba\Lambda{\Lambda'}$ is given by
\[\prod_p \Ba\Lambda{\Lambda'}_p.\]
Here  $p$ runs over all places of $\bQ$.

Cassels \cite{Cassels1998CasPairing} proved that this pairing is well-defined, namely it is independent of the choices of $P, Q_i$ and the representatives of the cosets of $\Lambda$ and $\Lambda'$.
Since skew-symmetry over $\bF_2$ is also symmetry,   the left kernel and the right kernel of the  Cassels pairing  are the same. To show its kernel, we first introduce some notation. From the short exact sequence $$0\to E[2]\to E[4]\oset{\times 2}\to E[2]\to 0,$$ we can derive the long exact sequence
\begin{eqnarray*}
0\to  E(\bQ)[2]/2E(\bQ)[4] \to \Sel_2(E)  \to \Sel_{4}(E) \to
\Im \Sel_{4}(E)\to 0.
\end{eqnarray*}
Cassels showed that the kernel of this  pairing   is
$\Im\Sel_4(E)/E(\bQ)[2]$.  The following lemma shows that    almost all the local Cassels pairings are trivial.
\begin{lemma}[Cassels \cite{Cassels1998CasPairing} Lemma 7.2]\label{lem-Cas} The local Cassels pairing $\Ba{\Lambda }{\Lambda' }_p=+1$ if $p$ satisfies
\begin{enumerate}
\item[(1)] $p\not=2,\infty$;
\item[(2)] The coefficients of $H_i$ and $L_i$ are all integral at $p$ for $ 1\le i\le 3$;
\item[(3)] Modulo $D_\Lambda$ and $L_i$ by $p$, they define a curve of genus $1$ over $\bF_p$ together with tangents to it.
\end{enumerate}
\end{lemma}

\subsection{ Residue Symbols}\quad

In this subsection, we will introduce several residue symbols. The first residue symbol is the additive Jacobi symbol. Let $d$ be a positive odd integer and $m$ an  integer coprime to $d$. Then we define the additive Jacobi symbol $\ALeg md=1$ if the Jacobi symbol $\Leg md=-1$ and $0$ otherwise.

Since other residue symbols involve the Gaussian integer ring $\bZ[i]$, we first recall some conceptions related to $\bZ[i]$.  A prime element $\lambda$ of $\bZ[i]$ is called Gaussian if it is not a rational prime. An integer $\theta\in\bZ[i]$ is called primary if  $\theta \equiv 1 \pmod{2+2i}$. In particular,  any primary integer can be written uniquely as the product of primary primes. We use $N$ to denote the norm of an element or an ideal of $\bZ[i]$.

The second residue symbol is the general Legendre symbol over $\bZ[i]$.   Let $\fp$ be a  prime ideal  of $\bZ[i]$  coprime to  $(1+i)$. The general Legendre symbol $\Leg{\alpha}\fp$   is  the unique element of $\Brbig{0, \pm1}$ such that
$\alpha^{\frac{N\fp-1}2}\equiv \Leg\alpha\fp \pmod \fp.$
We refer to   Page 196 of Hecke \cite{heckeGTM77}. In particular, if $\lambda$ is the unique primary prime in $\fp$, we  put $\Leg\alpha\lambda=\Leg\alpha\fp.$
If $\lambda$ has a factorization $\prod_{l=1}^k \lambda_l$ of primary primes, then we define $\Leg\alpha\lambda$ to be $\prod_{l=1}^k\Leg\alpha{\lambda_l}.$

The third residue symbol is the quartic residue symbol. We refer to Ireland-Rosen \cite{ireland1982classical}. Assume that $\lambda$ is a prime element   coprime to $(1+i)$. For a  Gaussian integer $\alpha$,   the quartic residue symbol $\Leg{\alpha}{\lambda}_4$   is defined to be  the unique element of $\Brbig{0, \pm 1, \pm i }$  such that
$\alpha^{\frac{N\lambda-1}4} \equiv\Leg\alpha\lambda_4 \pmod \lambda$.
% with  $\overline\lambda$   the   conjugate of $\lambda$.
Let $\lambda_1$ and $\lambda_2$ be two coprime Gaussian primes.
%$\Leg{\lambda_1}{\bar{\lambda_2}}_4 \Leg{\bar{\lambda_1}}{\lambda_2}_4=1$
We have the quartic reciprocity law
\[\Leg{\lambda_1}{\lambda_2}_4=\Leg{\lambda_2}{\lambda_1}_4 (-1)^{\frac{N\lambda_1-1}4\frac{N\lambda_2-1}4}.\]
Assume that $\lambda$ has a  factorization $\prod_{l=1}^k \lambda_l$ of primary primes. We define $\Leg\alpha\lambda_4$ to be $\prod_{l=1}^k\Leg\alpha{\lambda_l}_4.$

The last residue symbol is the rational quartic residue symbol. Let
$p$ be a rational prime congruent to $1$ modulo $4$. So there are exactly  two primitive  primes $\lambda$ and  $\bar\lambda$ lying  above $p$. Here $\overline\lambda$ is the complex conjugate of $\lambda$ and  $p=\lambda\overline\lambda$. If $q$ a rational integer such that $\Leg qp=1$, then  the two quartic residue symbols $\Leg q\lambda_4$ and $\Leg q{\bar\lambda}_4$ take the same value, and  we use the symbol $\Leg qp_4$ to denote any of them.  Moreover, if $d$ is a positive integer such that all prime factors of $d$ are congruent to $1$ modulo $4$, then $\Leg qd_4$ is defined to be \[\prod_{p\mid d} \Leg qp_4^{v_p(d)}\] provided that  $q$ is a rational integer satisfying $\Leg qp=1$ for any   $p\mid d$. Here
$v_p(d)$ is the $p$-adic valuation of $d$.

\subsection{Gauss Genus Theory}\quad

In this subsection, we briefly summarize  Gauss genus theory. One can refer to \S 3 of \cite{wang2016-2} for a detailed proof.

Let $K$ be an imaginary quadratic number field with ideal class group $\cA$. We write the multiplication of ideal classes additively. Then the $2^i$-rank $h_{2^i}(\cA)$ of $\cA$ is defined to be $\dim_{\bF_2} 2^{i-1}\cA/2^i\cA$ with $i$ any positive integer.
By classical Gauss genus theory, the $2$-rank $h_2(\cA)=t-1$ with $t$ the number of different prime factors of the fundamental discriminant $D$ of $K$. In fact, $h_2(\cA)$ equals to the dimension of $\cA[2]$, which is the set consisting of ideal classes killed by $2$. In addition, $\cA[2]$ is an elementary abelian $2$-group generated by $(p,\alpha_0)$ with $p$ any prime factor of $D$ and $2\alpha_0=D+\sqrt D$.

As to the $4$-rank, we can easily deduce that $h_4(\cA)=\dim_{\bF_2} 2\cA\cap \cA[2]$. Therefore, the study of the $4$-rank is reduced to that of $2\cA\cap\cA[2]$. The key tool to study  $2\cA\cap\cA[2]$ is the R\'edei matrix. They are closely tied  via $\cD(K)\cap N_{K/\bQ}(\Kc)$. Here $\cD(K)$ is the set of positive square-free divisors of $D$ and $N_{K/\bQ}$ is the norm map from $K$ to $\bQ$. In fact, we have a two to one epimorpism \[\theta: \cD(K)\cap N_{K/\bQ}(\Kc) \longrightarrow 2\cA\cap \cA[2]\]
with $\theta(d)=[(d,\alpha_0)]$. Note that $\cD(K)$ is a group under the group operation $d_1\odot d_2=\frac{d_1d_2}{(d_1,d_2)^2}$. Moreover, the kernel of $\theta$ is $\Brlr{1,D'}$ with $D'$ the square-free part of $D$.

To connect $\cD(K)\cap N_{K/\bQ}(\Kc)$ with the R\'edei matrix, we first introduce some notation. Let $p_1,\cdots,p_t$ be the different prime divisors of $D$. We assume that $p_t=2$ if $2\mid D$. The R\'edei matrix $R=(r_{ij})_{(t-1)\times t}$ of $K$ is an $\bF_2$ matrix defined by $r_{ii}=\ALeg{D/p_i^*}{p_i}$ and $r_{ij}=\ALeg{p_j}{p_i}$ if $i\not=j$. Here $p_i^*=(-1)^{\frac{p_i-1}2}p_i$.
\begin{lemma}\label{lem-gauss}
Assume that $w$ is a positive odd integer satisfying $(w,D)=1$ and $\Leg Dp=1$ for any prime divisor $p$ of $w$. Let $W=\brBig{\ALeg w{p_1},\cdots,\ALeg w{p_{t-1}}}^{\rm T}$. Then we have an isomorphism
\[\Brbig{d\in\cD(K) \big| dw\in N_{K/\bQ}(\Kc) } \longrightarrow \Brbig{ Y\in \bF_2^t \big|\; RY=W},\]
where the map is given by $d\mapsto Y_d:=\brbig{v_{p_1}(d),\cdots,v_{p_t}(d)}^{\rm T}$ and its inverse is $(y_1,\cdots,y_t)^{\rm T}\mapsto \prod_1^t p_i^{y_i}$.
\end{lemma}
Choose $w=1$ in this lemma, we obtain the following isomorphism
\[\cD(K)\cap N_{K/\bQ}(\Kc)\longrightarrow \Brbig{ Y\in \bF_2^t \big| \;RY=0}.\]
Consequently, $h_4(\cA)=t-1-\rank_{\bF_2}R$.

Like the $4$-rank, the study of the  $8$-rank $h_8(\cA)$   is equivalent to that
of $4\cA\cap \cA[2]$. This is equivalent to determine which $[\fa]\in 2\cA\cap
\cA[2]$ still lies in $4\cA$. We assume that $K=\bQ(\sqrt{-n})$, where
$n=p_1\cdots p_k$ is a positive square-free odd integer.
% with every $p_i$ congruent to $1$ modulo $4$.
Assume that $2^rd$ lies in $\cD(K)\cap N_{K/\bQ}(\Kc)$ such that $d$ is a non-trivial divisor of $n$ and $r$ equals to $0,1$. Then   the following equation
\begin{equation}\label{eq:gauss}
dx^2+\frac nd y^2=2^rz^2
\end{equation}
has a non-trivial integer solution.
\begin{lemma}\label{lem-gauss-h-8}
Assume that $n,d,r,R$ are as above and $(u,v,w)$ is a positive primitive integer solution to (\ref{eq:gauss}). Let $W=\brBig{\ALeg w{p_1},\cdots,\ALeg w{p_{k}}}^{\rm T}$. Then $[\fa]\in 4\cA$ if and only if there is a $Y\in \bF_2^{k+1}$ such that
$RY=W$, where $\fa=(2^rd,\alpha_0)$.
\end{lemma}

\subsection{Analytic Results}\quad

Given a  number field $F$,   let $n$ and $\cO$ be its degree and ring of algebraic integers respectively. We define $\Delta$ and   $N_F$ to be the   discriminant of $F$ and  the norm from $F$ to $\bQ$ respectively. We call a non-zero element $\gamma\in F$   totally positive if  it is positive under all real embedding provided that $F$ has a real embedding. If $F$ has no real embedding,   all non-zero elements of $F$ are  totally positive.

Let $\dag$ be  an integral ideal of $\cO$. We denote by $I(\dag)$   the group of all the fractional ideals that are coprime to $\dag$. We use $P_\dag$ to denote the group consisting of   the  principal fractional ideals $(\gamma)$ such that $\gamma$ is totally positive and   $\gamma\equiv 1\pmod\dag$.  Here the notation $\gamma\equiv 1\pmod \dag$   denotes $\gamma\in \cO_\fp$ and $\gamma\equiv 1 \pmod{ \fp^{v_{\fp}(\dag)}}$ for every prime ideal $\fp\mid\dag$, where $\cO_\fp$ is the integer ring of $F_\fp$.

If $\chi$ is a character of $I(\dag)/P_\dag$ with $\dag$ an integral ideal, then  we view it as a character on $I(\dag)$ call $\chi$   a character modulo $\dag$. In addition, if a fractional ideal $\fa$ is not coprime to $\dag$,   we define $\chi(\fa)=0$. Let $\Lambda(\fa)$  be the Mangoldt function defined by
\begin{equation*}
\begin{cases}
\log N_F\fp& \text{if $\fa=\fp^m$ with $m\ge 1$ },\\
0 & \text{ otherwise}.
\end{cases}
\end{equation*}
Then  $\psi(x,\chi)$ is defined to be
\[\psi(x,\chi)=\sum_{N_F\fa\le x} \chi(\fa) \Lambda(\fa).\]
The following     explicit formula (Proposition \ref{mainthm-explicitformula}) of $\psi(x,\chi)$  is proved in  P114 of Iwaniec-Kowalski \cite{iwaniec2004analytic}.
\begin{prop}\label{mainthm-explicitformula}\quad\\
If $\chi$ is a non-principal character modulo an integral ideal $\dag$ and $1\le T\le x$, then \begin{equation}\label{eq:expl-iwane}
\psi(x,\chi)=-\sum_{|\Im\rho|\le T}\frac{x^\rho-1}\rho+
O_F\brBig{xT^{-1}\cdot\log x \cdot \log(x^n\cdot N_F\dag)}.
\end{equation}
Here $\rho$ runs over all the zeros of $L(s,\chi)$ with $0\le \Re\rho \le 1$ and
$O_F$ means the implied constant only depends on $F$.
\end{prop}

Note that  the first term of the formula (\ref{eq:expl-iwane}) is not estimated. It can be estimated by the same way as  the classical case, and we omit its proof.
We  derive the   explicit  formula
\begin{equation}\label{eq:explicit-we need}
\psi(x,\chi)=-\frac{x^{\beta'}}{\beta'}+ R(x,T)
\end{equation}
with
\[R(x,T)\ll x\cdot\log^2(x  N_F\dag)\cdot\exp\brBig{-\frac{c_1\log x}{\log |T  N_F\dag|}}+xT^{-1}\log x\cdot\log \ABbig{x^n   N_F\dag}+x^{\frac14}\log x.\]
Here $c_1$ is a positive constant and the term $-\frac{x^{\beta'}}{\beta'}$ occurs only if $\chi$ is a real character such that $L(s,\chi)$ has a zero $\beta'$  satisfying
\[\beta'>1-\frac {c_2}{\log N_F\dag}\]
with $c_2$  a positive constant.

For further application, we  introduce  Siegel Theorem   and Page  Theorem  over $F$. The following Proposition \ref{thm-Siegel} is Siegel  Theorem over $F$, and the references are  Fogels \cite{Fogels1963, Fogels1965, Fogels1968}.
\begin{prop} \label{thm-Siegel}\quad
Let $\chi$ be a character modulo  an integral $\dag$  and $D=|\Delta| N_F\dag>D_0>1$.
\begin{enumerate}
\item[(i)] There is a positive constant $c_3=c_3(n)$   such that in the region
\[\Re (s)> 1-\frac {c_3}{\log D(1+\ABlr{\Im (s)})}>\frac34 \]
there is no zero of $L(s,\chi)$ if $\chi$ is complex, and for at most one real character $\chi' $ there maybe a simple zero $\beta'$ of $L(s,\chi')$.
\item[(ii)] Let $\beta'$ be the exceptional zero of the exceptional
character $\chi'$ modulo $\dag$. Then for any $\epsilon>0$ there exists a positive constant $c_4=c_4(n,\epsilon)$ such that
    \[1-\beta'> c_4(n,\epsilon) D^{-\epsilon}.\]
\end{enumerate}
\end{prop}

Proposition \ref{thm-Page} is   Page  Theorem over $F$. One can refer to Hoffstein-Ramakrishnan \cite{hoffstein1995siegel}.
\begin{prop} \label{thm-Page}\quad
For any $Z\ge 2$  and $c_5$  a suitable constant,   there is at most a
real primitive character $\chi$ to a modulus $\dag$ with $N_F\dag\le Z$ such that $L(s,\chi)$ has a real zero $\beta$ satisfying $$\beta >1-\frac{c_5}{\log Z}.$$
\end{prop}

\section{Representation of $2$-Selmer Group}
% \subsection{Pure $2$-Selmer group and unique representative }\quad

Now we apply the result of \S2.1 to the elliptic curve
\[\cE_{a,b}: y^2=x(x-a)(x+b)\]
with $a$ and $b$ positive odd integers. For  positive square-free integer $n$, we consider the elliptic curve $\cEn_{a,b}=\cEn$  $$\cEn: y^2=x(x-an)(x+bn).$$
Choosing $a_1=an, a_2=-bn$ and $ a_3=0$ in (\ref{D-Lambda-general}), we get the following identification
\[\Sel_2(\cEn)=\BrBig{\Lambda=(d_1,d_2,d_3)\in (\bQc/\bQ^{\times2})^{3}\; \Big| \;d_1d_2d_3\in\bQ^{\times2},\; D_\Lambda(\bA)\not=\emptyset}\]
with $D_\Lambda$ the genus one curve defined by\[
\left\{\begin{aligned}
H_1:& -bnt^2+d_2u_2^2-d_3u_3^2=0,\\
H_2:& -ant^2+d_3u_3^2-d_1u_1^2=0,\\
H_3:&\;2cnt^2+d_1u_1^2-d_2u_2^2=0.
\end{aligned}\right.\]Here $a+b=2c$.
If $(x,y)\in\cEn(\bQ)$ is not a $2$-torsion point, it corresponds to $(x-an,x+bn,x)$. Moreover, the four elements $(an,0),\; (-bn,0),\;(0,0)$ and  $O$  of $\cEn(\bQ)[2]$ correspond to $\big(2ac, 2cn, an \big),\; \big(-2cn, 2bc, -bn \big),\;
(-an,bn,-ab)$ and $(1,1,1)$ respectively.

In this section, we always assume  that  $a,b$ and $c $ are coprime  positive odd integers such that
\begin{equation}\label{cd-a,b}
a+b=2c.
\end{equation}

\subsection{Local Solvability Conditions on $D_\Lambda$}\quad

We denote by $a=a_1a_2^2,\; b=b_1b_2^2$ and  $c=c_1c_2^2$ with $a_1, b_1$ and $c_1$ square-free integers.
\begin{lemma}\label{lem-cd-Sel_2}
Assume that  $n$ is a  positive square-free   integer   coprime to $2abc$. Let $\Lambda=(d_1,d_2,d_3)$ with $d_i$ non-zero square-free integers such that $d_1d_2d_3$ is a square.
\begin{enumerate}
\item[(1)] If  $p\nmid 2abcn $, then $D_\Lambda(\bQ_p)\not=\emptyset$   if and only if $p\nmid d_1d_2d_3$.
\item[(2)] $D_\Lambda(\bR)$ is non-empty if and only if $d_2>0$.
\item[(3)] If $D_\Lambda(\bQ_2) $ is non-empty, then    $d_1$ and $d_2$ have the same parity.
\item[(4)] If $d_1$ and $d_2$ are odd, then $D_\Lambda(\bQ_2)$ is non-empty if and only if either $4\mid d_1-d_3,\; 8\mid d_1-d_2$ or $4\mid d_1+an,\;  8\mid d_1-d_2+2cn$.
\end{enumerate}
\end{lemma}
\begin{proof}
(1)  Let $p$ be a prime such that $p\nmid 2abcn$. If $p\mid d_1d_2d_3$, then  $p$ divides exactly two of $d_1, d_2$ and $d_3$. We assume that $p\mid d_1$ and $p\mid d_2$. Since we are dealing with homogeneous equations, we may assume that   $u_1, u_2, u_3$ and $t$ are $p$-adic integers and at least one of them is a $p$-adic unit. By comparing $p$-adic valuations of both sides of $H_3$, we get that $p\mid t$. Similarly, from $H_1$   we derive that $p\mid u_3$. Then via $H_1$ and $H_2$ we infer that $p\mid u_2$ and $p\mid u_1$. So $p\mid (t,u_1,u_2,u_3)$, which is impossible. Thus $D_\Lambda(\bQ_p)$ is empty provided that $p\mid d_1d_2d_3$.

Now we assume that $p\nmid d_1d_2d_3$. We will use Weil conjecture for curve (see Silverman \cite{SilvermanGTM106} P134) to show that   $D_\Lambda(\bQ_p)$ is non-empty.
Note that $D_\Lambda$ modulo $p$ gives rise to a smooth projective curve over $\bF_p$ by $p\nmid 2nabcd_1d_2d_3$. Weil conjecture implies that
\[Z(D_\Lambda,T)=\frac{P_1(T)}{(1-T)(1-pT)}\]
with $P_1(T)\in\bZ[T]$ factoring as
\[P_1(T)=(1-\alpha  T)(1-\overline\alpha T)\]
over $\bC$. Here  $\alpha$ has norm $p^\half$ and $Z(D_\Lambda,T) $ is the zeta function of $D_\Lambda$ over $\bF_p$ given by
\[Z(D_\Lambda,T)=\exp\brBig{\sum_{m=1}^\infty \#D_\Lambda(\bF_{p^m})\frac{T^m}m},\]
where $\#D_\Lambda(\bF_{p^m})$ denotes the number of points of $D_\Lambda $ over $\bF_{p^m}$. Therefore,
\[\sum_{m=1}^\infty \#D_\Lambda(\bF_{p^m}) \frac{T^m}m=
\log(1-\alpha T)+\log(1-\overline\alpha T)-\log(1-T)-\log(1-pT).\]
Comparing the coefficients of $T$ in above equation, we have
\[\#D_\Lambda(\bF_p)=1+p-(\alpha+\overline\alpha)\ge 1+p-2\sqrt p>0.\]
Therefore, $D_\Lambda(\bQ_p)$ is non-empty by Hensel's Lemma.

(2) If $d_2<0$, we see that $H_1(\bR)$ is not solvable for the coefficients  on its left hand side are all negative. Conversely, if $d_2>0$, then $D_\Lambda(\bR)$ is obviously solvable.

(3) Now we assume that $D_\Lambda(\bQ_2)$ is non-empty.  Assume that $d_1$ and $d_2$ have different parity. We first treat the case that $d_1$ is even  and $d_2$ is odd. Then $d_3$ is even by $d_1d_2d_3\in\bQ^{\times2}$.  Considering the $2$-adic valuations of both sides of  $H_2$ and $H_3$, we get that $t$ and $u_2$ are even. From similar considerations on $H_1$ and $H_3$,  we derive that $u_3$ and $u_1$ are also even.   So $u_1,u_2,u_3$ and $t$ are even, which is impossible. Similar arguments shows that the case that $d_1$ is odd and $d_2$ is even is also impossible.

(4) Let $d_1$ and $d_2$ be odd integers. First, we assume that  $D_\Lambda(\bQ_2)$ is non-empty. Since we are dealing with homogeneous equations, we may assume that $u_1, u_2, u_3$ and $t$ are $2$-adic integers and at least one of them is odd. Viewing $H_3$ as a congruence modulo $4$, we see that $u_1$ and $u_2$ are odd. From $H_2$ we infer that exactly one of $t$ and $u_3$ is even. Now we divide this into two subcases according to the parity of $t$.\\
Case (i): $t$ is odd. Then $u_3$ is even. Considering $H_2$ and $H_3$ as congruences modulo $4$ and $8$ respectively, we get that $4\mid d_1+an$ and $8\mid d_1-d_2+2cn$.\\
Case (ii): $t$ is even. Then $u_3$ is odd. Viewing $H_2$ and $H_3$ as congruences modulo $4$ and $8$ respectively, we obtain that $4\mid d_1-d_3$ and $8\mid d_1-d_2$.

Finally, we assume that either $4\mid d_1-d_3,\; 8\mid d_1-d_2$ or $4\mid d_1+an,\;  8\mid d_1-d_2+2cn$. According to these conditions, we divide it into two subcases.\\
Case (iii):   $4\mid d_1-d_3$ and $8\mid d_1-d_2$. We have $d_3\equiv d_1d_2\equiv1\pmod 8$. Since $4\mid d_1-d_3$, we get $d_1\equiv1\pmod4$. If $d_1\equiv1\pmod 8$, then $d_i$ has square root in $\bZ_2$ and  we choose $t=0$ and $u_i=d_i^{-\half}$ for $1\le i\le 3$. If $d_1\equiv 5\pmod8$, then $d_1+4an\equiv 1\pmod 8$ and $d_1+8cn\equiv d_2\pmod 8$; so $\brlr{d_1+8cn}{d_2}^{-1}$ and $\brlr{d_1+4an}{d_3}^{-1}$ have square roots in $\bZ_2$; in addition, we choose $t=2, u_1=1,$
$$u_2=\sqrt{\brlr{d_1+8cn}{d_2}^{-1}}   \quad {\rm and}\quad u_3=\sqrt{\brlr{d_1+4an}{d_3}^{-1}}.$$
Case (iv): $4\mid d_1+an$ and $8\mid d_1-d_2+2cn$. If $8\mid d_1+an$, then we choose $t=1$ and $u_3=0$; so $-and_1^{-1} $ and $ (-an+2cn)d_2^{-1} $ are congruent to $1$ modulo $8$, and they have square roots in $\bZ_2$;   let $u_1$ and $u_3$ be any square roots of $-and_1^{-1} $ and $ (-an+2cn)d_2^{-1} $ respectively. If $4 \parallel d_1+an$, we choose $t=1$ and $u_3=2$; so $(4-an)d_1^{-1}$ and $(2cn+4-an)d_2^{-1}$ are congruent to $1$ modulo $8$, and they have square roots in $\bZ_2$; we define $u_1$ and $u_2$ to be any square roots of $(4-an)d_1^{-1}$ and $(2cn+4-an)d_2^{-1}$ respectively.
Therefore, in any case we have $D_\Lambda(\bQ_2)\not=\emptyset$.

This completes the proof of the lemma.
\end{proof}

Assume that $n$ is a  positive square-free integer   coprime to $2abc$. By Lemma \ref{lem-cd-Sel_2},  any element of $\Sel_2(\cEn)/\cEn(\bQ)[2]$ has a unique representative $\Lambda=(d_1,d_2,d_3)$ with $d_i$ positive square-free integers satisfying
\begin{equation}\label{cd-Lambda-condition}
d_1d_2d_3\in\bQ^{\times2},\; d_2>0,\;   d_1| nabc\; {\rm{and}}\; d_2|nabc.
\end{equation}

\begin{lemma}\label{lem-Sel2-n part}
Assume that $n$ is a  positive square-free integer   coprime to $2abc$ and $\Lambda=(d_1,d_2,d_3)$ with $d_i$ positive square-free integers such that (\ref{cd-Lambda-condition}) holds. Let $p$ be a prime divisor of $n$.
\begin{itemize}
\item  If $p\nmid d_1$ and $p\nmid d_2$, then $D_\Lambda(\bQ_p)\not=\emptyset$   if and only if $\Leg{d_1}p=\Leg{d_2}p=1$.
\item If $p\nmid d_1$ and $p\mid d_2$, then $D_\Lambda(\bQ_p)$ is non-empty if and only if  $\Leg{d_1}p=\Leg{2ac}p$ and $\Leg{N/d_2}p=\Leg{2ab}p$, where $N=abcn$.
\item If $p\mid d_1$ and $ p\nmid d_2$, then $D_\Lambda(\bQ_p)$ is non-empty if and only if  $\Leg{N/d_1}p=\Leg{-2ab}p$ and $\Leg{d_2}p=\Leg {2bc}p$.
\item If $p\mid d_1$ and $ p\mid d_2$, then $D_\Lambda(\bQ_p)$ is non-empty if and only if  $\Leg{N/d_1}p=\Leg{-bc}p$ and $\Leg{N/d_2}p=\Leg{ac}p$.
\end{itemize}
\end{lemma}
\begin{proof}
Assume that $p\nmid d_1d_2$. So $p\nmid d_3$. If $D_\Lambda(\bQ_p)$ is non-empty, then by $H_2$ and $H_3$ we get $\Leg{d_2d_3}p=\Leg{d_1d_3}p=1$, namely $\Leg{d_1}p=\Leg{d_2}p=1$. Conversely, if $\Leg{d_1}p=\Leg{d_2}p=1$, then we set $t=0$ and $u_i=d_i^{-\half}$ for $1\le i\le 3$. So $(t,u_1,u_2,u_3)$ lies in $D_\Lambda(\bQ_p)$. The remained cases can be proved similarly. This completes the proof of the lemma.
\end{proof}
\begin{lemma}\label{lem-Sel2-abc part}
Assume that $n$ is a  positive square-free integer   coprime to $2abc$ and $\Lambda=(d_1,d_2,d_3)$ with $d_i$ positive square-free integers such that (\ref{cd-Lambda-condition}) holds.
\begin{itemize}
\item   Let $p$ be a prime divisor of $a$. Then $D_\Lambda(\bQ_p)\not=\emptyset$ implies that $p\nmid d_2$. Under the condition $p\nmid d_2$, we have
\begin{itemize}
\item if $p\nmid d_1$, then $D_\Lambda(\bQ_p)$ is non-empty if and only if $\Leg{d_2}p=1$;
\item if $p\mid d_1$, then $D_\Lambda(\bQ_p)$ is non-empty if and only if $\Leg{bnd_2}p=1$ provided that $p\mid a_1$ with $p\nmid a_2$ and $\Leg{d_2}p=\Leg{bn}p=1$ provided that $p\mid a_2$.
\end{itemize}
\item  Let $p$ be a prime divisor of $b$. Then $D_\Lambda(\bQ_p)\not=\emptyset$ implies that $p\nmid d_1$. Under the condition $p\nmid d_1$, we have
\begin{itemize}
\item if $p\nmid d_2$, then $D_\Lambda(\bQ_p)\not=\emptyset$ if and only if $\Leg{d_1}p=1$;
\item if $p\mid d_2$, then $D_\Lambda(\bQ_p)$ is non-empty if and only if  $\Leg{-and_1}p=1$ provided that $p\mid b_1$ with $p\nmid b_2$ and $\Leg{-an}p=\Leg{d_1}p=1$ provided that $p\mid b_2$.
\end{itemize}
\item    Let $p$ be a prime divisor of $c$. Then $D_\Lambda(\bQ_p)\not=\emptyset$ implies that $p\nmid d_3$. Under the condition $p\nmid d_3$, we have
\begin{itemize}
\item if $p\nmid d_1d_2$, then  $D_\Lambda(\bQ_p)$ is non-empty if and only if $\Leg{d_3}p=1$;
\item if $p\mid d_1$ and $p\mid d_2$, then $D_\Lambda(\bQ_p)$ is non-empty if and only if  $\Leg{and_3}p=1$ provided that $p\mid c_1$ with $p\nmid c_2$ and $\Leg{an}p=\Leg{d_3}p=1$ provided that $p\mid c_2$.
\end{itemize}
\end{itemize}

\end{lemma}
\begin{proof}
Let $p$ be a prime divisor of $a$. Assume that $D_\Lambda(\bQ_p)$ is non-empty. If $p\mid d_2$, then $p$ divides exactly one of $d_1$ and $d_3$. We may assume that $p\mid d_1$. So $p\nmid d_3$. Via $H_2$ and $H_3$, we obtain that $p\mid u_3$ and $p\mid t$. Then $H_1$ and $H_2$ imply that $p\mid u_2$ and $p\mid u_1$. So $p\mid (t,u_1,u_2,u_3)$, which is impossible. Therefore, $D_\Lambda(\bQ_p)\not=\emptyset$ implies that $p\nmid d_2$.

Now we assume that $p\nmid d_1d_2d_3$. If $D_\Lambda(\bQ_p)\not=\emptyset$, then $H_2$ implies that $\Leg{d_1d_3}p=1$, namely $\Leg{d_2}p=1$. Conversely, if $\Leg{d_2}p=1$, then we choose $u_1=\frac{d_2}{(d_1,d_2)}$. We have
\begin{eqnarray*}
u_3^2&=&d_2+and_3^{-1}t^2=d_2+\Obig p,  \\
u_2^2&=&d_3+2cnd_2^{-1}t^2,
\end{eqnarray*}where $\Obig p$ is $wp$ with $w$ some $p$-adic integer.
The first equation is solvable for any given $t$ by $\Leg{d_2}p=1$. Denote by $\cR$ the quadratic residue classes modulo $p$. Then $\cR$ and $d_3+2cnd_2^{-1}\cR$ have non-empty intersections for their cardinalities are $\frac{p+1}2$. Thus we can choose $t$ such that $d_3+2cnd_2^{-1}t^2$ is a quadratic residue. So $D_\Lambda(\bQ_p)$ is non-empty by Hensel's Lemma.

Now we consider the case $p\mid d_1$ and $p\mid d_3$. Then $p\nmid d_2$. First, we treat the subcase $p\mid a_2$. If $D_\Lambda(\bQ_p)$ is non-empty, then $\Leg{d_2bn}p=1$ by $H_1$. In addition, dividing $H_2$ by $p$ we derive that $\Leg{d_1/d_3}p=1$, namely $\Leg{d_2}p=1$. Conversely, if $\Leg{d_2}p=\Leg{bn}p=1$, then we choose $u_1=\frac{d_2}{(d_1,d_2)}$. We get
\begin{eqnarray*}
u_3^2&=&d_2+a_2^2d_3^{-1}\cdot a_1nt^2=d_2+\Obig  p,  \\
u_2^2&=&d_3+2cnd_2^{-1}t^2=bnd_2^{-1}t^2+\Obig p.
\end{eqnarray*}
These equations are solvable for $\Leg{d_2}p=\Leg{bn}p=1$.
So $D_\Lambda(\bQ_p)$ is non-empty. Finally, we treat the subcase $p\mid a_1$ with $p\nmid a_2$. If $D_\Lambda(\bQ_p)$ is non-empty, we get $\Leg{d_2bn}p=1$ by $H_1$. Conversely, if $\Leg{d_2bn}p=1$, then we choose $u_1=\frac{d_2}{(d_1,d_2)}$. We obtain
\begin{eqnarray*}
u_3^2&=&d_2+a_1d_3^{-1}\cdot a_2^2nt^2,  \\
u_2^2&=&d_3+2cnd_2^{-1}t^2=bnd_2^{-1}t^2+\Obig p.
\end{eqnarray*}
The second equation is solvable for any $t$ by $\Leg{d_2bn}p=1$. Like the case $p\nmid d_1d_2d_3$, we know that there is a $t$ such that $d_2+a_1d_3^{-1}a_2^2nt^2$ is a quadratic residue modulo $p$. So $D_\Lambda(\bQ_p)$ is non-empty. Hence, we finish the proof of the case $p\mid a$.

For the case $p\mid bc$, we can prove similarly. This completes the proof of the lemma.
\end{proof}

\begin{lemma}\label{lem-sel_2-place 2}
Assume that $n$ is a  positive square-free integer   coprime to $2abc$ and $\Lambda=(d_1,d_2,d_3)$ with $d_i$ positive square-free odd integers such that (\ref{cd-Lambda-condition}) holds. If $D_\Lambda(\bQ_p)$ is non-empty for all odd primes $p$ and $p=\infty$  , then $D_\Lambda(\bQ_2)$ is also non-empty.
\end{lemma}
\begin{proof}
We only prove the case that $a,b$ and $c$ are squares, the general case can be proved similarly but the process is much more complicate. Now we assume that $a, b $ and $c$ are squares. Define $A=(a,d_1)$. Then Lemma \ref{lem-Sel2-abc part} implies that $\Leg{bn}A=\Leg{d_2}A=1$. In addition, $\Leg{2cn}A=1$ for $a+b=2c$. Since $b$ and $c$ are squares, we have
\begin{equation}\label{eq-sel2-2place-A}
\Leg nA=1,\quad \Leg{d_2}A=1, \quad \Leg2A=1.
\end{equation}
Denote by $B=(b,d_2)$ and $C=(c,d_1)=(c,d_2)$. Similarly we have
\begin{equation}\label{eq-sel2-2place-B}
\Leg{-n}B=1,\quad \Leg{d_1}B=1,\quad \Leg2B=1
\end{equation}
and
\begin{equation}\label{eq-sel2-2place-C}
\Leg nC=1,\quad \Leg{d_3}C=1, \quad\Leg{-1}C=1.
\end{equation}
Put $n=n_1n_2n_3n_4$ with $(d_1,n)=n_3n_4$, $(d_2,n)=n_2n_4$ and $n_4=(d_1,d_2,n)$. Then we get that $d_1=ACn_3n_4$, $d_2=BCn_2n_4$ and $d_3=ABn_2n_3$.
By Lemma \ref{lem-Sel2-n part}, we have
\[\begin{array}{cccc}
   &\Leg{AC}{n_1}=\Leg{n_3n_4}{n_1},   &\Leg{BC}{n_1}=\Leg{n_2n_4}{n_1},&\quad  \\
   &\Leg{AC}{n_2}=\Leg{2n_3n_4}{n_2},  &\Leg{BC}{n_2}=\Leg{2n_1n_3}{n_2},&\qquad (*)\\
   &\Leg{AC}{n_3}=\Leg{-2n_1n_2}{n_3},  &\Leg{BC}{n_3}=\Leg{2n_2n_4}{n_3},&\quad\\
   &\Leg{AC}{n_4}=\Leg{-n_1n_2}{n_4},   &\Leg{BC}{n_4}=\Leg{n_1n_3}{n_4}.&\quad
\end{array}\]

%Now we want to use these Jacobi symbols to obtain three identities,
%which are written in additive Jacobi symbols to make the expression simpler.
From the second identity of equation (\ref{eq-sel2-2place-A}) we have
\[\ALeg{n_2n_4}A=\ALeg{BC}A=\ALeg AB+\ALeg AC+\ALeg{-1}A\ALeg{-1}B\]
by the quadratic reciprocity law and $\Leg{-1}C=1$ of (\ref{eq-sel2-2place-C}).   Via the second identities of (\ref{eq-sel2-2place-B}) and (\ref{eq-sel2-2place-C}), we get
\begin{eqnarray*}
\ALeg{n_2n_4}A&=&\ALeg{n_3n_4}B+\ALeg{n_2n_3}C+\ALeg{-1}A\ALeg{-1}B,  \\
&=&\ALeg B{n_3n_4}+\ALeg C{n_2n_3}+\ALeg{-1}A\ALeg{-1}B+\ALeg{-1}B \ALeg{-1}{n_3n_4}.
\end{eqnarray*}
Here we have used $\Leg{-1}C=1$ and the quadratic reciprocity law. The third,   sixth, seventh  and the last identities of (*) imply that
\begin{eqnarray*}
\ALeg{n_2n_4}A&=&\ALeg{-1}{n_4}+\ALeg2{n_2n_3}+\ALeg A{n_2n_4}+\ALeg{n_2n_3}{n_4}+\ALeg{n_3n_4}{n_2}+  \\
&&\ALeg{n_2n_4}{n_3}+\ALeg{-1}A\ALeg{-1}B+\ALeg{-1}B\ALeg{-1}{n_3n_4}.
\end{eqnarray*}
By the quadratic reciprocity law we get
\begin{eqnarray*}
\ALeg{2}{n_2n_3}+\ALeg{-1}{n_4}&=&\ALeg{-1}A\ALeg{-1}B+\ALeg{-1}A\ALeg{-1}{n_2n_4}+
\ALeg{-1}B\ALeg{-1}{n_3n_4}+  \\
&&\ALeg{-1}{n_2}\ALeg{-1}{n_3}+
\ALeg{-1}{n_2}\ALeg{-1}{n_4}+\ALeg{-1}{n_3}\ALeg{-1}{n_4}.
\end{eqnarray*}
Expanding the residue symbols $\ALeg{-1}{n_in_4}=\ALeg{-1}{n_i}+\ALeg{-1}{n_4}$ for $i=2$ and $3$, we have
\begin{equation}\label{eq:Sel-2part-1 id}
\ALeg{2}{n_2n_3}+\ALeg{-1}{n_4}=\ALeg{-1}{An_3}\ALeg{-1}{Bn_2}+
\ALeg{-1}{n_4}\ALeg{-1}{ABn_2n_3}.
\end{equation}
Similarly, starting from the first identity of (*) we derive
\begin{equation}\label{eq:Sel-2part-2 id}
\ALeg{-1}{n_3}=\ALeg{-1}{An_4}\ALeg{-1}{Bn_1}+\ALeg{-1}{n_3}\ALeg{-1}{ABn_1n_4},
\end{equation}
and starting from the second identity of (*) we obtain
\begin{equation}\label{eq:Sel-2part-3 id}
\ALeg{-1}{Bn_4}=\ALeg{-1}{An_1}\ALeg{-1}{Bn_4}+\ALeg{-1}{n_2}\ALeg{-1}{ABn_1n_4}.
\end{equation}

Now we use the  identities (\ref{eq:Sel-2part-1 id}), (\ref{eq:Sel-2part-2 id}) and (\ref{eq:Sel-2part-3 id}) to prove the lemma. By Lemma \ref{lem-cd-Sel_2}, it suffices to show that either $8\mid d_1-d_2, 4\mid d_1-d_3$ or $8\mid d_1-d_2+2cn, 4\mid d_1+an$. Note that $d_1-d_2=Cn_4(An_3-Bn_2)$ and $d_1-d_2+2cn\equiv Cn_4(An_3-Bn_2+2n_1n_2n_3)\pmod 8$. Moreover, $d_1-d_3\equiv An_3(n_4-Bn_2) \pmod 4$ and $d_1+an\equiv n_3n_4(A+n_1n_2)\pmod 4$. So we reduce to showing that either
$8\mid An_3-Bn_2, 4\mid n_4-Bn_2$ or $8\mid An_3-Bn_2+2n_1n_2n_3, 4\mid A+n_1n_2$.

Since $A, B, n_2$ and $n_3$ are odd, we know that either $4\mid An_3-Bn_2$ or $4\mid An_3-Bn_2+2n_1n_2n_3$. First, we consider the case $4\mid An_3-Bn_2$. Then
\begin{equation}\label{eq:Sel-2part-4 id}
\ALeg{-1}{An_3}=\ALeg{-1}{Bn_2}.
\end{equation}
Substituting this into (\ref{eq:Sel-2part-1 id}) we get
\[\ALeg{2}{n_2n_3}=\ALeg{-1}{An_3n_4}=\ALeg{-1}{Bn_2n_4}.\]
Moreover, substituting (\ref{eq:Sel-2part-4 id}) into (\ref{eq:Sel-2part-2 id}) and (\ref{eq:Sel-2part-3 id}) we have
\begin{eqnarray*}
0&=&\ALeg{-1}{An_3n_4}\ALeg{-1}{Bn_1n_3}=\ALeg{-1}{Bn_2n_4}\ALeg{-1}{Bn_1n_3},  \\
0&=&\ALeg{-1}{ABn_1n_4}\ALeg{-1}{Bn_2n_4}=\ALeg{-1}n\ALeg{-1}{Bn_2n_4}.
\end{eqnarray*}
Adding these we deduce that $\ALeg2{n_2n_3}=\ALeg{-1}{Bn_2n_4}=0$. From the third identities of (\ref{eq-sel2-2place-A}) and (\ref{eq-sel2-2place-B}) we see that $\Leg{2}{An_3}=\Leg2{Bn_2}$. Thus $8\mid An_3-Bn_2$. Since $\ALeg{-1}{Bn_2n_4}=0$, we have $\Leg{-1}{n_4}=\Leg{-1}{Bn_2}$. So $4\mid n_4-Bn_2$.

Finally, we consider the case $4\mid An_3-Bn_2+2n_1n_2n_3$. Then we have
\begin{equation}\label{eq:Sel-2part-5 id}
\ALeg{-1}{ABn_2n_3}=1.
\end{equation}
Substituting this into (\ref{eq:Sel-2part-1 id}) we obtain $\ALeg2{n_2n_3}=0.$
From (\ref{eq:Sel-2part-2 id}) we derive that
\[\ALeg{-1}{ABn_1n_4}\ALeg{-1}{Bn_2n_4}=0.\]
Via (\ref{eq:Sel-2part-3 id}) we obtain
\[\ALeg{-1}{Bn_1n_3}=\ALeg{-1}{ABn_1n_4}\ALeg{-1}{Bn_1n_3}.\]
Adding these two equations we get $\ALeg{-1}{Bn_1n_3}=0$ by noting that $\ALeg{-1}{ABn_1n_4}=1+\ALeg{-1}n$. To prove $8\mid An_3-Bn_2+2n_1n_2n_3$, it suffices to show that $\ALeg{2}{An_3+2n_1n_2n_3}=\ALeg2{Bn_2}$. By the  supplementary law of the quadratic reciprocity law, we have
\[\ALeg{2}{An_3+2n_1n_2n_3}=\ALeg2{An_3}+\ALeg{-1}{An_1n_2}+1=\ALeg{2}{An_3}+
\ALeg{-1}{Bn_1n_3}=\ALeg2{An_3}.\]
Noting that $\ALeg2{n_2n_3}=0$ and  $\ALeg2{A}=\ALeg2B=0$, we have $8\mid An_3-Bn_2+2n_1n_2n_3$. Via $\ALeg{-1}{Bn_1n_3}=0$ and (\ref{eq:Sel-2part-5 id}), we have $\ALeg{-1}{An_1n_2}=1$, namely $4\mid A+n_1n_2$.

This completes the proof of the lemma.
\end{proof}

\subsection{Matrix Representation of $2$-Selmer Group}\quad

From Lemma \ref{lem-cd-Sel_2}, \ref{lem-Sel2-n part}, \ref{lem-Sel2-abc part} and \ref{lem-sel_2-place 2}, there is a matrix representation of the pure $2$-Selmer group $\Sel_2'(\cEn)=\Sel_2(\cEn)/\cEn(\bQ)[2]$. For our purpose, we only give this matrix representation under the condition that $a, b, c$ are squares and $n$ is a positive square-free odd integer such that all prime factors $p$ of $n$ satisfy
\begin{equation}\label{cd-n-prime div cong 1 mod abc}
 \Leg pq=1
\end{equation} with $q$ any prime divisor of $abc$.

To give the matrix representation of $\Sel_2'(\cEn)$, we first introduce some notation. Denote by $n=p_1\cdots p_k$. Assume that $a$, $b$ and $c$ have the prime decompositions \begin{equation}\label{cd-factor-abc}
q_1^{t_1}\cdots q_{k_1}^{t_{k_1}},\quad
q_{k_1+1}^{t_{k_1+1}}\cdots q_{k_2}^{t_{k_2}}\quad {\rm{and}} \quad
q_{k_2+1}^{t_{k_2+1}}\cdots q_{k_3}^{t_{k_3}}
\end{equation} respectively.
Here $k_1\le k_2\le k_3$   are non-negative integers, and all the $t_i$ are positive even integers. Note that every element of $\Sel_2'(\cEn)$ has a unique representative $(d_1,d_2,d_3)$ with $d_i$ positive integers satisfying (\ref{cd-Lambda-condition}). Since $d_1\mid nac$ and $d_2\mid nbc$, we put $d_1=p_1^{x_1}\cdots p_k^{x_k}q_1^{z_1}\cdots q_{k_1}^{z_{k_1}}q_{k_2+1}^{z_{k_2+1}}\cdots q_{k_3}^{z_{k_3}}$ and $d_2=p_1^{y_1}\cdots p_k^{y_k}q_{{k_1+1}}^{w_{{k_1+1}}}\cdots q_{k_3}^{w_{k_3}}$. Here $x, y, z$ and $w$ are row vectors over $\bF_2$ given by $(x_1,\cdots,x_k)$, $(y_1,\cdots,y_k)$, $(z_1,\cdots,z_{k_1},
z_{k_2+1}, \cdots,z_{k_3})$ and  $(w_{k_1+1},\cdots,w_{k_3})$ respectively.

First, we  define the matrix $M_1$  which gives rise to the matrix representation of $\Sel_2'(E)$. To this purpose, we first introduce some $\bF_2$   matrices. Let $F=(f_{ij})_{k_3\times k_3}$ be the matrix defined by $f_{ii}=0$ and $f_{ij}=\ALeg{q_j}{q_i}$ if $i\not=j$. We write $F$ in the following block matrix form
\[F=\left(
      \begin{array}{ccc}
        F_1 &F_2 & F_3 \\
        F_4 &F_5 & F_6 \\
        F_7& F_8 & F_9 \\
      \end{array}
    \right).
\]Here $F_1$ and $F_5$ have sizes $k_1\times k_1$ and $(k_2-k_1)\times(k_2-k_1)$ respectively.
Let $\Delta$ and $\Delta'$ be the diagonal matrices given by
\begin{eqnarray*}
\Delta&=&\diag\brbig{1,\cdots,1},\\
\Delta'&=&\diag\brlr{\ALeg{-1}{q_{k_1+1}},\cdots,\ALeg{-1}{q_{k_2}}},
\end{eqnarray*}where the size of $\Delta$ is $\brlr{k_3-k_2}\times \brlr{k_3-k_2}$.

Now we define $M_1$ to be the $(2k_3-k_1)\times(2k_3-k_2)$ matrix
\[\left(
    \begin{array}{cccc}
        &   & F_2 & F_3 \\
      F_4 & F_6 &   &   \\
      F_7 &   & F_8 &   \\
       &   & \Delta' &    \\
       & \Delta &   & \Delta \\
    \end{array}
  \right).
\]By Lemma \ref{lem-cd-Sel_2},   \ref{lem-Sel2-abc part} and \ref{lem-sel_2-place 2}, we know that the map $(d_1,d_2,d_3)\mapsto (z,w)$ induces an isomorphism
\[\Sel_2'(E)\longrightarrow \Brlr{(z,w):
M_1\left(
     \begin{array}{c}
       z^{\rm T} \\
       w^{\rm T} \\
     \end{array}
   \right)=0,\; z\in \bF_2^{k_1+k_3-k_2}, w\in \bF_2^{k_3-k_1}
 }.\]In fact, this can be verified by block matrices. Taking the first block row of $M_1\left(
      \begin{array}{c}
        z^{\rm T} \\
        w^{\rm T} \\
      \end{array}
    \right)=0
 $, we get $\left(
             \begin{array}{cc}
               F_2 & F_3 \\
             \end{array}
           \right)w^{\rm T}=0
$. This is $\sum_{k_1+1}^{k_3} w_j\ALeg{q_j}{q_i}=1$ for all $i\le k_1$. From this we obtain that $\Leg{d_2}{q_i}=1$ for all $i\le k_1$, which is compatible with the case of $q\mid a$ in Lemma \ref{lem-Sel2-abc part}. The remaining block rows  can be checked similarly. We have to remak on the last block  row of $M_1
\left(
      \begin{array}{c}
        z^{\rm T} \\
        w^{\rm T} \\
      \end{array}
    \right)=0
$. From this we obtain that $\Delta({z'}^{\rm T}+{w'}^{\rm T})=0$ with $z'=(z_{k_2+1},\cdots,z_{k_3})$ and $w'=(w_{k_2+1},\cdots,w_{k_3})$. Then $z_i=w_i$ for $i>k_2$, which is equivalent to $(c,d_3)=1$.\\

Now we use $M_1$ to give the matrix representation of $\Sel_2'(\cEn)$. We first introduce some notation. Let $A=A_n=(a_{ij})_{k\times k}$ be the $\bF_2$ matrix    given by $a_{ii}=\sum_{l\not=i} a_{il}$ and $a_{ij}=\ALeg{p_j}{p_i}$ if $i\not=j$. Denote by $G=(g_{ij})_{k\times k_3}$ the $\bF_2$ matrix defined by $g_{ij}=\ALeg{q_j}{p_i}$. We write $G$ in the block matrix form
\[G=\left(
      \begin{array}{ccc}
        G_1 & G_2 & G_3 \\
      \end{array}
    \right)
\]with the sizes of $G_1$ and $G_2$ being $k\times k_1$ and $k\times(k_2-k_1)$ respectively. Let $\cM_n$  be the Momsky matrix (see Appendix of Heath-Brown \cite{HeathBrownMonsky1993Selmergroup})
$$\Ma{A+D_{-2}}{D_2}{D_2}
{A+D_2},$$
where $D_u=\diag\brlr{\ALeg u{p_1},\cdots,\ALeg u{p_k}}$. We define $M_n$ to be the $\bF_2$ matrix
\[\Ma{\cM_n}{\cG}{}{M_1}\]
with\[\cG=
\left(
  \begin{array}{cccc}
    G_1 & G_3 &  &  \\
     &  & G_2 & G_3 \\
  \end{array}
\right).
\]
\begin{prop}\label{prop-sel2-rep}
Let $a, b, c$ be odd squares with factorization (\ref{cd-factor-abc})  and $n$ a positive square-free integer  satisfying (\ref{cd-n-prime div cong 1 mod abc}). Then the following is an isomorphism
\[\Sel_2'(\cEn) \longrightarrow \ker M_n,\quad (d_1,d_2,d_3)\mapsto (x,y,z,w)^{\rm T},\] where $(d_1,d_2,d_3)$ is   the representative of an element  of $\Sel_2'(\cEn)$ such that (\ref{cd-Lambda-condition}) holds. Here $x=\brbig{v_{p_1}(d_1),\cdots, v_{p_k}(d_1)}, y=\brbig{v_{p_1}(d_2),\cdots,v_{p_k}(d_2)}$, $w=\brbig{v_{q_{k_1+1}}(d_2),\cdots,v_{q_{k_3}}(d_2)}$ and $z=\brbig{v_{q_1}(d_1),\cdots,v_{q_{k_1}}(d_1),v_{q_{k_2+1}}(d_1),\cdots,v_{q_{k_3}}(d_1)}$.
\end{prop}
For any $(d_1,d_2,d_3)$ above, we put $d_1=d_1'd_1''$ with $d_1'=(d_1,n) $ and $d_1''=(d_1,abc)$. Similarly, we set $d_2=d_2'd_2''$ and $d_3=d_3'd_3''$. Since $n$ satisfies (\ref{cd-n-prime div cong 1 mod abc}), those $d_1', d_2', d_3'$ and $n$ occurred in the local solvability conditions for $p\mid abc$ (Lemma \ref{lem-Sel2-abc part}) vanish. Note that $d_1''$ and $d_2''$ correspond to $z$ and $w$ respectively. Combing these, we have $M_1
\left(
  \begin{array}{c}
    z^{\rm T} \\
    w^{\rm T} \\
  \end{array}
\right)=0
$. From $a, b$ and $c$ being squares,   those $ a, b$ and $c$ occurred in the local solvability conditions for $p\mid n$ (Lemma \ref{lem-Sel2-n part}) also vanish. Note that $d_1'$ and $d_2'$ correspond to $x$ and $y$ respectively. Observe the identity
\[x_i\sum_{l\not=i}\ALeg{p_l}{p_i}+ \sum_{j\not=i}
x_j\ALeg{p_j}{p_i}=x_i\ALeg{n/d_1'}{p_i}+
(1-x_i)\ALeg{d_1'}{p_i}.\]
From Lemma \ref{lem-Sel2-n part}, we get
\[x_i\ALeg{n/d_1'}{p_i}+
(1-x_i)\ALeg{d_1'}{p_i}=\ALeg{d_1''}{p_i}+x_i\ALeg{-2}{p_i}+y_i\ALeg2{p_i}.\]
Similar result also holds for $y$. From these we can derive that $$\cM_n
\left(
  \begin{array}{c}
    x^{\rm T} \\
    y^{\rm T} \\
  \end{array}
\right)+\cG
\left(
  \begin{array}{c}
    z^{\rm T} \\
    w^{\rm T} \\
  \end{array}
\right)=0.
$$ %Here we have used the quadratic reciprocity law for $\ALeg{\widehat{ac}}{p_i}$ and
%the condition (\ref{cd-n-prime div cong 1 mod abc}).
So the map is well-defined. Its injectivity is obvious; by Lemma \ref{lem-cd-Sel_2}, \ref{lem-Sel2-n part}, \ref{lem-Sel2-abc part} and \ref{lem-sel_2-place 2}, it is surjective.
% \newpage

\section{Torsion Subgroup}
In this section, we will prove that $\En_\tor(\bQ)\simeq \brbig{\bZ/2\bZ}^2$, where $\En$ is defined by (\ref{eq-En}) with $(a,b,c)$ any positive primitive integer solution to $a^2+b^2=2c^2$. Let $\cE$ be an elliptic curve with full $2$-torsion points. According to Mazur's classification theorem on torsion subgroup of elliptic curves over $\bQ$ (see \cite{SilvermanGTM106}), $\cE_\tor(\bQ)$ is isomorphic to $\bZ/2\bZ\times \bZ/2m\bZ$ for some $m=1,2,3,4$. Ono \cite{Ono} has the following characterization of $\cE_\tor(\bQ).$
\begin{lemma}[Ono]\label{lem-Ono}Let $\cE: y^2=x(x-a)(x+b)$ be an elliptic curve over $\bQ$ with $a,b$ integers. Then $\cE[2](\bQ)\simeq\brbig{\bZ/2\bZ}^2$.
\begin{enumerate}
\item[(1)] $\cE_\tor(\bQ)$ contains a point of order $4$ if and only if one of the three pairs $ [-a,b], [a,a+b]$ and $[-b,-a-b]$ consists of squares of integers.
\item[(2)] $\cE_\tor(\bQ)$ has a point of order $8$ if and only if there exist a positive integer $d$ and pairwise coprime integers $u, v$ and $w$ such that $u^2+v^2=w^2$ and $[d^2u^4,d^2v^4]$ is one of the three pairs in (1).
\item[(3)] $\cE_\tor(\bQ)$ has a point of order $3$ if and only if there exist a positive integer $d$ and pairwise coprime integers $u$ and $v$ such that
    $a=-(u^4+2u^3v)d^2,\;b=(v^4+2v^3u)d^2\quad {\rm{and}}\quad \frac uv\not\in\BrBig{-2,-\half,-1,1,0}.$
\end{enumerate}
\end{lemma}
\begin{prop}\label{prop-E-torsion}For any square-free integer $n$,
$\En(\bQ)\simeq \brbig{\bZ/2\bZ}^2$.
%The torsion subgroup $\En_\tor(\bQ)$ of $\En(\bQ)$ is isomorphic to $\brbig{\bZ/2\bZ}^2$.
\end{prop}
\begin{proof}
From the definition of $\En$, we see that $\En_\tor(\bQ)$ contains a subgroup isomorphic to $\brbig{\bZ/2\bZ}^2$. The proposition will be proved by showing that $\En_\tor(\bQ)$ contains no point of order $4$ and $3$. Here we have used Mazur's
classification theorem on torsion subgroup of elliptic curves over $\bQ$.

Now we show that $\En_\tor(\bQ)$ has no point of order $4$. Note that none of the three pairs $[-a^2n, b^2n], [a^2n, a^2n+b^2n=2c^2n]$ and $[-b^2n, -a^2n-b^2n]$ consists of squares of integers. So (1) of Lemma \ref{lem-Ono} implies that $\En_\tor(\bQ)$ contains no point of order $4$.

So we remain to prove that $\En_\tor(\bQ)$ contains no point of order $3$. By (3) of Lemma \ref{lem-Ono}, it suffices to show that the equation
\begin{equation*}
\begin{cases}
d^2u^3(u+2v)=&-a^2n,\\
d^2v^3(v+2u)=&b^2n
\end{cases}\qquad (*)
\end{equation*}with $d$ a positive integer and $u,v$ pairwise coprime integers is not solvable. We will divide the proof of this into several steps.

First, as $a$ and $b$ are coprime integers, we get that the greatest common divisor  $(-a^2n, b^2n)$ of $-a^2n$ and $b^2n$ is $|n|$, which is square-free. From equation (*), we get that $d^2\mid (-a^2n, b^2n)=|n|$. So $d=1$ by noting that $d$ is a positive integer.

Second, we claim that $|n|=(u^4+2u^3v, v^4+2v^3u)=(3,u-v)$. Since $(u,v)=1$, we have $(u,v+2u)=1$. Thus $(u^3,v^3(v+2u))=1$. Similarly, $(u+2v,v^3)=1$. Therefore, $(u^4+2u^3v, v^4+2v^3u)$ equals to
\[(u+2v,v^3(v+2u))=(u+2v,v+2u)=(u+2v,u-v)=(u-v,3u)=(3,u-v).\]
So $n$ divides $u+2v$ and $v+2u$. Therefore, the equation (*) is reduced to
\begin{equation*}
\begin{cases}
v^3\cdot \frac{v+2u}n=b^2, &\qquad (*1)\\
-u^3\cdot \frac{u+2v}n=a^2.& \qquad (*2)
\end{cases}
\end{equation*}

Next, we assume that $v>0$ to solve the equations (*1) and (*2).
Since $v$ is coprime to $v+2u$, from (*1) we derive that there are integers $b_1$ and $ b_2$ which divide $b$ and satisfy  \[v^3=b_1^2,\quad  {v+2u}=nb_2^2.\]
The first equation implies that $b_1=b_3^3$ for some $b_3\mid b$. Since $b$ is odd, there are odd integers $b_2$ and $b_3$ such that
\begin{equation}\label{eq-tor1}
v=b_3^2, \quad v+2u=nb_2^2.
\end{equation}
Now we claim that $u<0$. Otherwise  $u>0$,   from $v+2u=nb_2^2$ we get $n>0$ and
$$na^2=-u^3\cdot(u+2v)<0$$
by (*2). Then we have $a^2<0$, which is impossible. Hence, we get $u<0$. Like (*1), from (*2) we know that there are odd integers $a_2$ and $a_3$ such that
\begin{equation*}\label{eq-tor2}
u=-a_3^2,\quad u+2v=na_2^2.
\end{equation*}
Combing this with equation (\ref{eq-tor1}), we derive that there are odd integers $a_2, a_3, b_2$ and $b_3$ such that
\[b_3^2-2a_3^2=nb_2^2,\quad 2b_3^2-a_3^2=na_2^2.\]
Viewing these equations as congruences modulo $8$, we see that $n\equiv1\pmod 8$ and $n\equiv-1\pmod 8$ respectively. These can't be true at the same time. So the   equations (*1) and (*2) are not solvable if $v>0$.

Finally, the equations (*1) and (*2) are not solvable for $v<0$. This can be proved similarly as the case $v>0$.

Therefore, $\En_\tor(\bQ)$ contains no point of order $3$. This finishes the proof of the proposition.
\end{proof}

\section{Non-trivial Shafarevich-Tate Group}
In this section, we always assume that $(a,b,c)$ is a  positive primitive  integer solution to $a^2+b^2=2c^2$ and $n=p_1\cdots p_k\equiv1\pmod8$ is a positive square-free integer  satisfying $\Leg{p_i}q=1$ for any $1\le i\le k$ and prime divisor $q$ of $abc$. Note that $\cE=\cE_{a,b}=E_{a^2,b^2}$ and $\En(\bQ)[2]$ consists of
\[(2,2n,n),\; (-2n,2,-n),\; (-n,n,-1) \;{\rm and}\; (1,1,1).\] We  assume that $\Sel_2(\cE)$ has dimension two.

\subsection{Proof of Theorem \ref{thm-ab odd square-p +-1 mod 8-introd}}\quad

In this subsection, we always assume that all prime divisors $p_i$ of $n$ are congruent to $\pm1$ modulo $8$. Then the Monsky  matrix $\cM_n$ is of the form
\[\diag\brbig{A+ D_{-1},\; A}.\]
\begin{lemma}\label{lem-pm1 mod 8}
Assume that all prime divisors $p_i$ of $n$ are congruent to $\pm1$ modulo $8$. Let $x=(x_1,\cdots,x_k)^{\rm T}$ and $x_0=(1,\cdots,1)^{\rm T}$ be two column vectors in $\bF_2^k$. Denote by $d=\prod_{i=1}^kp_i^{x_i}$.
\begin{enumerate}
\item[(1)] $Ax=0$ if and only if $x^{\rm T}(A+D_{-1})=0$.
\item[(2)] If $(A+D_{-1})x=0$, then $x^{\rm T}A=0$ if $d\equiv1\pmod8$ and $(x_0-x)^{\rm T}A=0$ if $d\equiv-1\pmod8$.
\item[(3)] The dimension of the pure $2$-Selmer group $\Sel_2'(\En)$ is two if and only if $h_4(n)=1$. If this is satisfied, then $\Sel_2'(\En)$ is generated by
 $$ \Lambda=(2,2,1),\quad \Lambda'=(d,1,d).$$
\end{enumerate}
\end{lemma}
\begin{proof}
For notational simplicity, we assume that $p_1\equiv\cdots\equiv p_l\equiv1\pmod8$ and $p_{l+1}\equiv\cdots\equiv p_k\equiv-1\pmod 8$. From $n\equiv1\pmod8$ we derive that $k-l$ is even. In addition, we can divide the matrices $A, D_{-1}, x$ and $x_0$ into block matrices
\[A=\Ma{A_1}{A_2}{A_3}{A_4},\quad D_{-1}=\Ma{0}{}{}I,\quad x=
\left(
  \begin{array}{c}
    z \\
    w \\
  \end{array}
\right),\quad
x_0=\left(
  \begin{array}{c}
    z_0 \\
    w_0 \\
  \end{array}
\right).
\]Here the sizes of $A_4$ and $I=\diag(1,\cdots,1)$ are $(k-l)\times(k-l)$, and the sizes of $z$ and $z_0$ are $l\times1$. We hope that the matrices $A_1, A_2, A_3$ and  $A_4$ are not confused with $A_n$ when $n=1, 2, 3$ and $4$. From the definition of $A$ and $D_{-1}$, we obtain that
\begin{equation}\label{matrix-property-A-D}
A_1^{\rm T}=A_1,\;A_2^{\rm T}=A_3 \quad {\rm and}\quad  A_4^{\rm T}+A_4+I=E.
\end{equation}
Here $E$ is a $(k-l)\times(k-l)$ matrix with all components being $1$, and we have used the quadratic reciprocity law. Since $k-l$ is even, we get $Ew_0=0$.\\

(1) If $Ax=0$, then by the block forms of $A$ and $x$ we get that $A_1z+A_2w=0$ and $A_3z+A_4w=0$. While
\[x^{\rm T}(A+D_{-1})=
\left(
  \begin{array}{cc}
    z^{\rm T} & w^{\rm T} \\
  \end{array}
\right)
\Ma{A_1}{A_2}{A_3}{A_4+I}=
\left(
  \begin{array}{cc}
    z^{\rm T}A_1+w^{\rm T}A_3 & z^{\rm T}A_2+w^{\rm T}(A_4+I)  \\
  \end{array}
\right).
\]
Note that $(z^{\rm T}A_1+w^{\rm T}A_3)^{\rm T}=
A_1z+A_2w=0$ and $\Sqbig{z^{\rm T}A_2+w^{\rm T}(A_4+I)}^{\rm T}
=A_3z+(A_4^{\rm T}+I)w$. We claim that \emph{ $w$ has even many non-zero components.}
From  this claim we deduce that $Ew=0$, namely $(A_4^{\rm T}+I)w=A_4w$ by (\ref{matrix-property-A-D}). Therefore,
\[\Sqbig{z^{\rm T}A_2+w^{\rm T}(A_4+I)}^{\rm T}=A_3z+A_4w=0.\]
Thus $x^{\rm T}(A+D_{-1})=0$. Moreover, the above process is invertible.

So the proof of (1) is reduced to proving the claim. From the definition of $A$, we know that $Ax_0=0$, namely
\begin{equation}\label{matrix-Ax_0=0}
A_1z_0+A_2w_0=0\quad {\rm and}\quad A_3z_0+A_4w_0=0.
\end{equation}
Like the expansion of $x^{\rm T}(A+D_{-1})$, we get $x_0^{\rm T}A=
\left(
  \begin{array}{cc}
    0 & w_0^{\rm T} \\
  \end{array}
\right).
$ Therefore, $0=x_0^{\rm T}Ax=
\left(
  \begin{array}{cc}
    0 & w_0^{\rm T} \\
  \end{array}
\right)x=w_0^{\rm T} w$, which is equivalent to the claim.\\

(2) If $d\equiv1\pmod8$, then $w$ also has even many non-zero components and the proof is the same as that of (1). We assume that $d\equiv-1\pmod 8$, namely $w$ has odd many non-zero components. Then
\begin{equation}\label{matrix-Ew=w_0}
Ew=w_0=Iw_0.
\end{equation}
Denote by $\barz=z_0-z$ and $\barw=w_0-w$. So $(x_0-x)^{\rm T}=
\left(
  \begin{array}{cc}
    \barz^{\rm T} & \barw^{\rm T}  \\
  \end{array}
\right)$. From $(A+D_{-1})x=0$ we obtain that $A_1z+A_2w=0$ and $A_3z+(A_4+I)w=0$.
Observe that
\[(x_0-x)^{\rm T}A=
\left(
  \begin{array}{cc}
    \barz^{\rm T} A_1+\barw^{\rm T}A_3 &\barz^{\rm T}A_2+\barw^{\rm T}A_4 \\
  \end{array}
\right).
\]But we have
$\brbig{\barz^{\rm T} A_1+\barw^{\rm T}A_3}^{\rm T}=A_1\barz+A_2\barw=
A_1z_0+A_2w_0-(A_1z+A_2w)=0$ by (\ref{matrix-Ax_0=0}). Moreover,
\begin{eqnarray*}
\brbig{\barz^{\rm T}A_2+\barw^{\rm T}A_4}^{\rm T}&=&
A_3\barz+A_4^{\rm T}\barw=A_3z_0+A_4^{\rm T}w_0+ A_3z+A_4^{\rm T}w \\
&=&A_3z_0+(A_4+I)w_0+A_3z+(A_4+I)w+Ew  \\
&=& Iw_0+Ew=0.
\end{eqnarray*}
Here we have used (\ref{matrix-Ax_0=0}) and (\ref{matrix-Ew=w_0}). Consequently, $(x_0-x)^{\rm T}A=0$.\\

(3) By Proposition \ref{prop-sel2-rep}, to find all the elements of $\Sel_2'(\En)$ is equivalent to compute the kernel of $M_n$. This is equivalent to find all $(X,Y,Z,W)$ such that
$$\cM_n\left(
           \begin{array}{c}
             X \\
             Y \\
           \end{array}
         \right)
+\cG\left(
           \begin{array}{c}
             Z \\
             W \\
           \end{array}
         \right)=0 \quad
{\rm and} \quad M_1\left(
           \begin{array}{c}
             Z\\
             W \\
           \end{array}
         \right)=0.$$
Here $X, Y\in \bF_2^k$, $Z\in\bF_2^{k_1+k_3-k_2}$ and $W\in\bF_2^{k_3-k_1}$. Since $\Sel_2(\cE)$ has dimension two, we know
that $\ker M_1=\Brlr{0}$ by Proposition \ref{prop-sel2-rep}.
This implies that $Z$ and $W$ are zero vectors. Thus we  reduce to finding $X$ and $Y$ in $\bF_2^k $ such that $\cM_n\left(
           \begin{array}{c}
             X \\
             Y \\
           \end{array}
         \right)=0$.
Then $\dim_{\bF_2}\Sel_2'(\En)=2k-\rank\cM_n$.

From (1) we derive that $\rank A=\rank (A+D_{-1})$. Hence, $\rank \cM_n=2\rank A$. But the R\'edei matrix $R$ of $\bQ(\sqrt{-n})$ takes the form $
\left(
  \begin{array}{cc}
    A & 0 \\
  \end{array}
\right)$. Therefore, $\rank R=\rank A=k-h_4(n)$. Thus $h_4(n)=1$ if and only if $\rank\cM_n=2k-2$. Hence, the dimension of $\Sel_2'(\En)$ is two if and only if $h_4(n)=1$.

Now we assume that  $h_4(n)=1$ to find representatives of $\Sel_2'(\En)$. From above argument, it suffices to find $X, Y\in\bF_2^k$ such that
$\cM_n\left(
        \begin{array}{c}
          X \\
          Y \\
        \end{array}
      \right)=0
$, namely $(A+D_{-1})X=0$ and $AY=0$.
As the rank of $A$ is $k-1$ and $Ax_0=0$, we get $Y=0$ or $x_0$. Similarly, $X=0$ or $x$. Hence, $\Sel_2'(\En)$ is generated by $(1,n,n)$ and $(d,1,d)$. Noting that $\Lambda=(2,2,1)=(1,n,n)(2,2n,n)$, the proof of (3) is complete.

This completes the proof of the lemma.
\end{proof}

Now we can prove Theorem \ref{thm-ab odd square-p +-1 mod 8-introd}.\\
\emph{Proof of Theorem \ref{thm-ab odd square-p +-1 mod 8-introd}.}
From the exact sequence (\ref{sel-2 condition}), we see that the a necessary condition for (1) is $\dim_{\bF_2}\Sel_2'(\En)=2$. This is equivalent to $h_4(n)=1$ by Lemma \ref{lem-pm1 mod 8}. Now we assume this and use the notation of Lemma \ref{lem-pm1 mod 8}. Since  $h_4(n)=1$, from (3) of Lemma  \ref{lem-pm1 mod 8} we know that $\Sel_2'(\En)$ is generated by $\Lambda=(2,2,1)$ and $\Lambda'=(d,1,d)$.

First, we compute the Cassels pairing $\Ba\Lambda{\Lambda'}$.  The genus one curve $D_\Lambda$ is
\[\begin{cases}
H_1:& -b^2nt^2+2u_2^2-u_3^2=0,\\
H_2:& -a^2nt^2+u_3^2-2u_1^2=0,\\
H_3:&c^2nt^2+u_1^2-u_2^2=0.
\end{cases}\]
From the definition of the Cassels pairing, we have to choose points on $H_i(\bQ)$. For $H_3$ we choose $Q_3=(0,1,1)$. So the corresponding tangent linear form $L_3$ of $H_3$ at $Q_3$ is \[L_3: u_1-u_2.\]
For $H_1$ we will use Gauss genus theory to select a point. As $
R\left(
   \begin{array}{c}
     0 \\
     1 \\
   \end{array}
 \right)=
\left(
   \begin{array}{cc}
     A&0 \\
   \end{array}
 \right)
\left(
   \begin{array}{c}
     0 \\
     1 \\
   \end{array}
 \right)=0$, by Gauss genus theory (Lemma \ref{lem-gauss}) $2$ is a norm element, namely there is a positive primitive integer solution $(\alpha,\beta,\gamma)$ to
\begin{equation}\label{root-pm1-n}
\alpha^2+n\beta^2=2\gamma^2.
\end{equation}
Let $Q_1=(\beta,\gamma b,\alpha b)$. Then $Q_1$ lies in $H_1(\bQ)$ and the corresponding tangent linear form is
\[L_1: \beta bnt-2\gamma u_2+\alpha u_3.\]
So by Lemma \ref{lem-Cas}, $\Ba\Lambda{\Lambda'}$ equals to
\[\prod_{p\mid 2nabc\infty} \brBig{L_1L_3(P_p),d}_p.\]
Here $P_p$ is any point on $D_\Lambda(\bQ_p)$ such that $L_i(P_p)$ is non-vanishing.
For any prime divisor $p$ of $abc$,  we have  $\Leg{p_i}p=1$; so we have $\Leg dp=1$. The local Cassels pairing is trivial at all $p\mid abc$. In addition, it is also trivial at $p=\infty$ by $d>0$. So we only need to compute it at $p=2$ and $p\mid n$.

For $p\mid n$, we choose the local point $(t,u_1,u_2,u_3)$ such that
\[t=0, u_1=1, u_2=-1, u_3^2=2.\]
As $(2\gamma+\alpha u_3)(2\gamma-\alpha u_3)=4\gamma^2-2\alpha^2=2n\beta^2$, we choose $u_3$ such that $p\mid 2\gamma-\alpha u_3$. Then $p\nmid 2\gamma+\alpha u_3$ by the primitivity of $(\alpha,\beta,\gamma)$. So
\[\brBig{L_1L_3(P_p),d}_p=\brBig{2(2\gamma+\alpha u_3),d}_p=\Leg\gamma p^{\delta}.\]
Here $\delta$ is $1$ if $p\mid d$ and $0$ otherwise.

For $p=2$, we note that the local Cassels pairing is trivial if $d\equiv1\pmod8$. So we only need to consider the case $d\equiv-1\pmod 8$. Let $(t,u_1,u_2,u_3)$ be the local point satisfying\[t=1, u_1=0, u_2^2=c^2n, u_3^2=a^2n.\]
We observe that
\begin{eqnarray*}
(\beta bn+\alpha u_3)(\beta bn-\alpha u_3)&=&\beta^2b^2n^2-\alpha^2a^2n
  =b^2n(2\gamma^2-\alpha^2)-\alpha^2a^2n\\
&=&2\gamma^2 b^2n-2\alpha^2c^2n=2b^2n\brBig{\gamma^2-\alpha^2\frac{c^2}{b^2}}
\equiv0\pmod{16}.
\end{eqnarray*}
So we may choose $u_3$ such that $8\mid \beta bn+\alpha u_3$. Then
\[\brBig{L_1L_3(P_p),d}_2=\brBig{-u_2(\beta bn-2\gamma u_2+\alpha u_3),-1}_2
=\brBig{2\gamma u_2^2,-1}_2=\Leg{-1}\gamma.\]
So we get $$\Ba\Lambda{\Lambda'}=\Leg\gamma d \;{\rm or}\; \Leg\gamma d\Leg{-1}\gamma$$
according to   $d\equiv1\pmod 8$ or not.

Next, we show that (1) is equivalent to $h_4(n)=1$ and the non-degeneracy of the Cassels pairing on $\Sel_2'(\En)$. From the following short exact sequence $$0\longrightarrow \En[2]\longrightarrow \En[4]\oset{\times2}\longrightarrow \En[2]\longrightarrow0,$$
we obtain the derived long exact sequence
\[0\longrightarrow \En(\bQ)[2]/2\En(\bQ)[4]\longrightarrow
\Sel_2(\En)\longrightarrow\Sel_4(\En)\longrightarrow\Im\Sel_4(\En)\longrightarrow0.\]
By Proposition \ref{prop-E-torsion},  we see that (1) is equivalent to $\dim_{\bF_2}\Sel_2'(\En)=2$ and $\#\Sel_2(\En)=\#\Sel_4(\En)$, which are equivalent to $\dim_{\bF_2}\Sel_2'(\En)=2$ and $\#\Im\Sel_4(\En)=4$ by the long exact sequence. Noting that the kernel of the Cassels pairing on $\Sel_2'(\En)$ is $\Im\Sel_4(\En)/\En(\bQ)[2]$, we infer that (1) is equivalent to $h_4(n)=1$ and the non-degeneracy of the Cassels pairing on $\Sel_2'(\En)$ by Lemma \ref{lem-pm1 mod 8}.

Finally, we  connect the Cassels pairing with $h_8(n)$. We first assume that $d\equiv1\pmod 8$. As $2\mid 2n$ is a norm element satisfying (\ref{root-pm1-n}),  $h_8(n)=1$ if and only if $W\in \Im R$ with $W=\brBig{\ALeg \gamma{p_1},\cdots,\ALeg\gamma{p_k}}^{\rm T}$ by Lemma \ref{lem-gauss-h-8}. Via (2) of Lemma \ref{lem-pm1 mod 8}, we see that $x^{\rm T}A=0$. As $A$ has rank $k-1$ and $R=
\left(
  \begin{array}{cc}
    A & 0 \\
  \end{array}
\right)
$, we obtain that $W\in \Im R$ if and only if $x^{\rm T}W=0$, namely $\Leg \gamma d=1$. Hence, the Cassels pairing is non-degenerate  if and only if $h_8(n)=0$ provided that $d\equiv1 \pmod 8$.

Now we assume that $d\equiv-1\pmod8$. Then $(x_0-x)^{\rm T}A=0$ by (2) of Lemma \ref{lem-pm1 mod 8}. Like the above case,  $h_8(n)=1$ if and only if $(x_0-x)^{\rm T}W=0$, namely $\Leg{\gamma}{n/d}=1$. Viewing (\ref{root-pm1-n}) as a congruence modulo $\gamma$, we have $\Leg{-n}\gamma=1$. Since $n\equiv1\pmod 8$, we have $\Leg{-1}\gamma=\Leg \gamma n$ from the quadratic reciprocity law. So the Cassels pairing $\Ba\Lambda{\Lambda'}$ is $\Leg\gamma{n/d}$ in this case. Hence, the Cassels pairing is non-degenerate if and only if $h_8(n)=0$.

In summary, we have proved that (1) is equivalent to (2). This completes the proof of the theorem.
\qed

\subsection{Proof of Theorem \ref{thm-ab odd square-p 1 mod 4-introd}}
\begin{lemma}\label{lem-p  1 mod4}Assume that all prime divisors of $n$ are congruent to $1$ modulo $4$. Then the dimension of the pure $2$-Selmer group $\Sel_2'(\En)$   is two if and only if $h_4(n)=1$.
\end{lemma}
\begin{proof}Like (3) of Lemma \ref{lem-pm1 mod 8},
it suffices  to show that $\rank\cM_n=2k-2$ if and only if $h_4(n)=1$.  Note that the Monsky matrix takes the form
\[\cM_n=\Ma{A+D_{2}}{D_{2}}{D_{2}}{A+D_{2}}.\]
Here $D_{2}=\diag\brBig{\ALeg2{p_1},\cdots,\ALeg2{p_k}}$. To relate $\cM_n$ to the R\'edei matrix $R$ of $n$, we perform some elementary linear transforms on the block matrix $\cM_n$. Adding the first block row to the second, we have
\[\Ma{A+D_{2}}{D_{2}}{A}{A}.\]
Then adding the second block column to the first, we get
\[\Ma{A}{D_{2}}{}{A}.\]
Summating all the last $(k-1)$ columns to the $(k+1)$-th column, we derive
\[\Ma R{D_{2}'}{}{A'}.\]
Here $D_{2}'$ and $A'$ denote the matrices obtained from $D_{2}$ and $A$ respectively by deleting their first columns. Adding all the first $(k-1)$ rows to the $k$-th row and then moving the $k$-th row as the last row, we yield
\begin{equation}\label{matrix-Rk R1^T}
\Ma{R_k}*{}{R_1^{\rm T}}.
\end{equation}
Here $R_i$ is the matrix obtained from the R\'edei matrix $R$ by deleting its $i$-th row. Since every $p_i$ is congruent to $1$ modulo $4$ and $n$ is congruent to $1$ modulo $8$, we see that the column sum of $R$ is zero, namely the sum of any of $R$'s given column is zero. Thus\[\rank R_i=\rank R=k-h_4(n).\]From  this and (\ref{matrix-Rk R1^T}) we get
\[2\rank R\le \rank \cM_n\le k-1+\rank R.\]
Therefore, $\rank R=k-1$ if and only if $\rank\cM_n=2k-2$. Then the lemma follows from $\rank R=k-h_4(n)$.
\end{proof}

\begin{lemma}\label{lem-p 1 mod 4 rep of Sel}Assume that all prime divisors of $n$
are congruent to $1$ modulo $4$ and $h_4(n)=1$.
%Let $n$ be a positive square-free integer congruent to $1$ modulo $8$ with all prime
%factors congruent to $1$ modulo $4$ such that $h_4(n)=1$.
\begin{enumerate}
\item[(1)]  If the rank of $A$ is
$k-2$, then $\Sel_2'(\En)$ is generated by $(n,n,1)$ and $(d,d,1)$,
% \[(n,n,1)\; {\rm{and}}\;  (d,d,1),\]
where $d=\prod_1^k p_i^{x_i}$ with $x=(x_1,\cdots,x_k)^{\rm T}\not=0, x_0=(1,\cdots,1)^{\rm T}\in\bF_2^k$ satisfying $Ax=0$.
\item[(2)] If the rank of $A$ is $k-1$, then $\Sel_2'(\En)$ is generated by
$(n,n,1)$ and $(d,n/d,n)$,
% \[(n,n,1)\; {\rm{and}}\;  (d,n/d,n),\]
where $d=\prod_1^k p_i^{x_i}$  with $x=(x_1,\cdots,x_k)^{\rm T}$ and $\fb=\brBig{\ALeg2{p_1},\cdots,\ALeg2{p_k}}^{\rm T}$ satisfying $Ax=\fb$.
\end{enumerate}

\end{lemma}
\begin{proof}Like (3) of Lemma \ref{lem-pm1 mod 8},  we reduce to finding $X$ and $Y$ in $\bF_2^k$ such that
$\cM_n\left(
           \begin{array}{c}
             X \\
             Y \\
           \end{array}
         \right)=0$,
namely
\begin{equation}\label{solve-YZ-p 1 mod 4}
\begin{cases}
 A X+D_{2}(X+Y)&=0,\\
A Y+D_{2}(X+Y)&=0.
\end{cases}
\end{equation}
Adding these two equations, we get $A(X+Y)=0$. So $X+Y$ lies in $\ker A$.
According to $\rank A_n$, we can divide this into two cases.

First, we deal with the case $\rank A=k-1$. Then $\ker A=\Brbig{ 0, x_0 }$. If  $X+Y=0$, then $X=Y\in \ker A$ by (\ref{solve-YZ-p 1 mod 4}). These give rise to two elements $X=Y=0$ or $X=Y=x_0$. So $(1,1,1)$ and $(n,n,1)$ lie in $\Sel_2'(\En)$. If $X+Y=x_0$, then (\ref{solve-YZ-p 1 mod 4}) implies that \[A X=D_{2}x_0=\fb.\]
In fact, $\fb$ is indeed in the image of $A$. From Gauss genus theory, we know that
\[\rank
\left(
  \begin{array}{cc}
    A & \fb \\
  \end{array}
\right)
=k-h_4(n)=k-1=\rank A.\]
Let $x$ be such that $Ax=\fb$. Then $X=x$ or $x_0-x$. So these give rise to the remaining two elements $(d,n/d,n)$ and $(n/d,d,n)$ of $\Sel_2'(\En)$.

Finally, we consider the case $\rank A=k-2$. If $X+Y=0$, then $X=Y\in \ker A$ by (\ref{solve-YZ-p 1 mod 4}).  Thus there are four elements in $\Sel_2'(\En)$ given by
\[(n,n,1),\;(d,d,1),\;(n/d,n/d,1)\;{\rm and}\; (1,1,1)\]with $d$ defined in the lemma. This completes the proof of the lemma.
\end{proof}

Now we can prove Theorem \ref{thm-ab odd square-p 1 mod 4-introd}.\\
\emph{Proof of Theorem \ref{thm-ab odd square-p 1 mod 4-introd}.} Like the proof of Theorem \ref{thm-ab odd square-p +-1 mod 8-introd}, a necessary condition for (1) is $h_4(n)=1$. So we may assume that  $h_4(n)=1$. Then there are two cases according to the rank of $A$.

First, we consider the case $\rank A=k-2$. We use the notation defined in (1) of  Lemma \ref{lem-p 1 mod 4 rep of Sel}.
% Let $x=(x_1,\cdots,x_k)^{\rm T}\not=0, x_0:=(1,\cdots,1)_k^{\rm T}$ such that $Ax=0$.
Then we have
$R\left(
    \begin{array}{c}
      x \\
      0 \\
    \end{array}
  \right)=
\left(
  \begin{array}{cc}
    A & \fb \\
  \end{array}
\right)
\left(
    \begin{array}{c}
      x \\
      0 \\
    \end{array}
  \right)=Ax=0
$. Here $\fb=\brlr{\ALeg2{p_1},\cdots,\ALeg2{p_k}}^{\rm T}$. So by Gauss genus theory (Lemma \ref{lem-gauss} and   the epimorphism $\theta$ in \S2.4),
$d_0=\prod_1^k p_i^{x_i}$ corresponds to the non-trivial element of $2\cA\cap\cA[2]$. So $d=d_0=\prod_1^k p_i^{x_i}$. We claim that $d\equiv 5\pmod8$. Since $h_4(n)=1$, we know that
$\rank \left(
  \begin{array}{cc}
    A & \fb \\
  \end{array}
\right)=k-1$ by Gauss genus theory. While $\rank A=k-2$, we see that $\fb$ is not in the image of $A$. As $A$ is symmetric and $\ker A$ is generated by $x_0$ and $x$, a column vector $y\in\bF_2^k$ is in the image of $A$ if and only if $x_0^{\rm T}y=x^{\rm T}y=0$.
For the vector $\fb$, we have $x_0^{\rm T}\fb=\ALeg2n=0$
for $n\equiv1\pmod 8$. Thus $x^{\rm T}\fb\not=0$. Note that $x^{\rm T}\fb=\ALeg2d$. Therefore, $d$ is congruent to $5$ modulo $8$.

By Lemma  \ref{lem-p  1 mod4},   the pure $2$-Selmer group $\Sel_2'(\En)$ has dimension two. Moreover, via Lemma \ref{lem-p 1 mod 4 rep of Sel} it is generated by
\[\Lambda=(d,d,1),\; \Lambda'=(-1,1,-1).\]
Note that $\Lambda'=(n,n,1)(-n,n,-1)$. Now we compute the Cassels pairing
$\Ba\Lambda{\Lambda'}.$
For $\Lambda=(d,d,1)$, the genus one curve $D_\Lambda$ associated to $\Lambda$ is given by\[\begin{cases}
H_1:& -b^2nt^2+du_2^2-u_3^2=0,\\
H_2:& -a^2nt^2+u_3^2-du_1^2=0,\\
H_3:&2c^2nt^2+du_1^2-du_2^2=0.
\end{cases}\]
By Cassels  pairing, we have to choose global points on $H_i$. For $H_3$, we choose the global point $Q_3=(0,1,1)\in H_3(\bQ)$. Then the corresponding tangent linear form $L_3$ of $H_3$ at $Q_3$ is  \[L_3: u_1-u_2.\]
Now we are ready to choose a point on $H_1(\bQ)$. Since $d$ corresponds to the non-trivial element of $2\cA\cap \cA[2]$, Gauss genus theory implies that $d$ is a norm element, namely there is a positive primitive integer solution $(\alpha,\beta,\gamma)$ to
\begin{equation}\label{root-h_4(n) p 1 mod 4}
d\alpha^2+d'\beta^2=\gamma^2
\end{equation}with $dd'=n$. We may assume that $\alpha$ is even. If not,   $\alpha$ is odd and $\beta$ is even. We can get a new solution
\[(\bar\alpha,\bar\beta,\bar\gamma)=\brBig{ d'\alpha-2d'\beta-d\alpha, d\beta-2d\alpha-d'\beta, (d+d')\gamma}\] to (\ref{root-h_4(n) p 1 mod 4}).
Thus $4\mid \bar \alpha$ and $2\parallel \bar\beta$. Dividing the solution $(|\bar\alpha|,|\bar\beta|,|\bar\gamma|)$ by $\gcd(\bar\alpha,\bar\beta,\bar\gamma)$, we derive a positive primitive integer solution to (\ref{root-h_4(n) p 1 mod 4}) with corresponding $\alpha$ being even.
From (\ref{root-h_4(n) p 1 mod 4}), we know that $(\beta,\gamma  b,d\alpha  b)$ lies in $H_1(\bQ)$ with the corresponding tangent linear form
\[L_1: d'\beta  b t-\gamma u_2+\alpha u_3.\]

By Lemma \ref{lem-Cas}, the Cassels pairing $\Ba\Lambda{\Lambda'}$ equals to
\[\prod_{p|2nabc\infty} \brBig{L_1L_3(P_p),-1}_p.\]
Here   $P_p$ is any point of   $D_\Lambda(\bQ_p)$ such that $L_i(P_p)$ is non-vanishing.
Since any prime divisor $p$ of $n$ is congruent to $1$ modulo $4$, we know that  the local Cassels pairing is trivial at $p\mid n$. This pairing is also trivial at $p\mid c$ for $-a^2=b^2-2c^2\equiv b^2\pmod p$. Thus we reduce to computing the local Cassels pairing at $p\mid 2ab\infty$.

For $p=\infty$, we choose the local point $P_\infty=(t,u_1,u_2,u_3)=(0,1,-1,\sqrt d)$.
Then
\[\brBig{L_1L_3(P_\infty),-1}_\infty=\brBig{2(\gamma+\alpha\sqrt d),-1}_\infty=1.\]
For $p=2$, we choose the local point $P_2=(t,u_1,u_2,u_3)$ such that
\[t=2,\;u_2=1,\;u_1^2=1+8c^2d',\;u_3^2=d-4b^2n.\]
Here we used the fact that $d\equiv5\pmod8$. As $u_1^2=1+8c^2d'\equiv9\pmod{16}$, we may assume that $u_1\equiv 3\pmod 8$. Since $\alpha$ is even, we have $2\parallel \alpha$ by (\ref{root-h_4(n) p 1 mod 4}).  Let $u_3$ be any square root of $d-4b^2n$. Then $2\mid d'\beta b+\frac\alpha2 u_3$. The local Cassels pairing at $p=2$ is
\begin{eqnarray*}
\brBig{L_1L_3(P_2),-1}_2&=&\brBig{(u_1-1)(2d'\beta b+\alpha u_3-\gamma),-1}_2  \\
&=&(2,-1)_2 \brBig{2\brbig{d'\beta b+\frac\alpha2 u_3}-\gamma,-1}_2  \\
&=&(-\gamma,-1)_2=-\Leg{-1}\gamma.
\end{eqnarray*}
For $p\mid ab$, we choose the local solution
\[t=0,\;u_1=1,\;u_2=-1,u_3^2=d.\]
By (\ref{root-h_4(n) p 1 mod 4}), $(\gamma+\alpha u_3)(\gamma-\alpha u_3)=\gamma^2-\alpha^2 u_3^2=d'\beta^2$. If $p\mid \beta$, then we choose $u_3$ such that $p\mid\gamma-\alpha u_3$; so $p\nmid \gamma+\alpha u_3$. If $p\nmid \beta$, we choose $u_3$ to be any square root of $d$ and  have $p\nmid \gamma-\alpha u_3$. Thus in any case $\gamma+\alpha u_3$  is a $p$-adic unit.
So the local pairing is
\[\brBig{L_1L_3(P_p),-1}_p=\brbig{2(\gamma+\alpha u_3),-1}_p=1.\]
Therefore, the Cassels pairing  $\Ba\Lambda{\Lambda'}$ is \[\Ba\Lambda{\Lambda'}=-\Leg{-1}\gamma.\]
Like Theorem \ref{thm-ab odd square-p +-1 mod 8-introd}, (1) is equivalent to $h_4(n)=1$ and  the non-degeneracy of the Cassels pairing.

Now we claim that $\Leg{-1}\gamma=1$ if and only if $h_8(n)=1$ provided that $h_4(n)=1$. As $d=d_0\mid 2n$ corresponds to the non-trivial element of $2\cA\cap \cA[2]$ and $(d\alpha)^2+n\beta^2=d\gamma^2$ by (\ref{root-h_4(n) p 1 mod 4}), we see that $h_8(n)=1$ if and only if $W\in \Im R$ by Lemma \ref{lem-gauss-h-8}. Here $W=\brlr{\ALeg \gamma{p_1},\cdots,\ALeg\gamma{p_k}}^{\rm T}$ and $R$ is the R\'edei matrix $
\left(
  \begin{array}{cc}
    A & \fb \\
  \end{array}
\right)
$. Since the rank of $R$ is $k-1$ and $x_0^{\rm T}R=0$,   a vector $z$ is in the image of $R$ if and only if $x_0^{\rm T}z=0$. We have $x_0^{\rm T}W=\ALeg\gamma n$. Consequently, $h_8(n)=1$ if and only if $\Leg\gamma n=1$. Viewing (\ref{root-h_4(n) p 1 mod 4}) as a congruence modulo $\gamma$, we see that $\Leg{-n}\gamma=1$. By $n\equiv1\pmod4$ and the quadratic reciprocity law, we have $\Leg{-1}\gamma=\Leg\gamma n$. Therefore, $h_8(n)=1$ if and only if $\Leg{-1}\gamma=1$.

So in  the case $\rank A=k-2$, (1) is equivalent to (2) by noting that $d\equiv5\pmod8$.\\

Now we consider the case that $\rank A=k-1$. We use the notation of (2) of Lemma \ref{lem-p 1 mod 4 rep of Sel}. Note that
$R\left(
    \begin{array}{c}
      x \\
      1 \\
    \end{array}
  \right)=
\left(
  \begin{array}{cc}
    A & \fb \\
  \end{array}
\right)
\left(
    \begin{array}{c}
      x \\
      1 \\
    \end{array}
  \right)=Ax+\fb=0
$. By Gauss genus theory, $d_0=2\prod_1^k p_i^{x_i}$ corresponds to the non-trivial element of $2\cA\cap\cA[2]$. Then $d=\prod_1^k p_i^{x_i}$. From Lemma \ref{lem-p 1 mod 4 rep of Sel} we derive that $\Sel_2'(\En)$ is generated by
\[\Lambda=(2d,2d,1),\quad \Lambda'=(-1,1,-1).\]
Here we used the fact that $\Lambda=(d,n/d,n)(2,2n,n)$ and $\Lambda'=(-n,n,-1)(n,n,1)$.

Now we begin to compute the Cassels pairing $\Ba\Lambda{\Lambda'}$. The genus one curve $D_\Lambda$ is defined by
\[\begin{cases}
H_1:& -b^2nt^2+2du_2^2-u_3^2=0,\\
H_2:& -a^2nt^2+u_3^2-2du_1^2=0,\\
H_3:&c^2d't^2+u_1^2-u_2^2=0.
\end{cases}\] Here $d'=n/d$.
According to the definition of the Cassels pairing, we first choose global points $Q_i$ on $H_i(\bQ)$. Let $Q_3=(0,1,1)$. Then $Q_3$ lies in $H_3(\bQ)$ and the corresponding tangent linear form $L_3$ is
\[L_3:   u_1-u_2.\]
Since $2d$ corresponds to the non-trivial element of $2\cA\cap\cA[2]$, by Gauss genus theory there is a positive primitive integer solution $(\alpha,\beta,\gamma)$ to
\begin{equation}\label{root-2d-rank k-1 case}
d\alpha^2+d'\beta^2=2\gamma^2.
\end{equation}
Then $Q_1=(\beta,\gamma b,\alpha bd)$ lies in $H_1(\bQ)$. In addition, the tangent linear form $L_1$ of $H_1$ at $Q_1$ is
\[L_1: d'\beta  bt-2\gamma u_2+\alpha u_3.\]

Like the case $\rank A=k-2$, the Cassels pairing $\Ba\Lambda{\Lambda'}$ equals to
\[\prod_{p\mid 2ab\infty}\brBig{L_1L_3(P_p),-1}_p\]
with $P_p$ any point on $D_\Lambda(\bQ_p)$ such that   $ L_i(P_p)$ is non-vanishing.
For $p=\infty$, we choose the local point $(t,u_1,u_2,u_3)=(0,1,-1,\sqrt{2d})$. Then
\[\brBig{L_1L_3(P_\infty),-1}_\infty=\brBig{2 (\alpha\sqrt{2d}+2\gamma),-1}_\infty=1.\]
For $p=2$, we choose a point $P_2=(t,u_1,u_2,u_3)$ such that
\[t=1, u_1=2\ALeg2d, u_2^2=c^2d'+u_1^2, u_3^2=a^2n+2du_1^2\]
with $\gamma u_2\equiv 1\pmod 4$. Note that
\begin{eqnarray*}
&&(d'\beta b+\alpha   u_3)(d'\beta  b-\alpha  u_3)={d'}^2\beta^2b^2-\alpha^2(a^2n+2du_1^2)  \\
&=&d'b^2(2\gamma^2-d\alpha^2)-\alpha^2(a^2n+2du_1^2)=2d'b^2\SqBig{\gamma^2-\alpha^2
\frac1{2d'b^2}\brlr{2c^2n+2du_1^2}}  \\
&=&2d'b^2\SqBig{\gamma^2-\alpha^2\brBig{\frac {c^2}{b^2}d+ \frac{d}{b^2d'}u_1^2}}\equiv0\pmod{16}.
\end{eqnarray*}
So we may choose $u_3$ such that $8\mid d'\beta  b+\alpha   u_3$. We have
\begin{eqnarray*}
\brBig{L_1L_3(P_2),-1}_2&=&\brBig{(u_1-u_2)(d'\beta  b+\alpha   u_3-2\gamma   u_2),-1}_2  \\
&=&\brBig{-2\gamma u_2(u_1 -u_2),-1}_2  \\
&=&(2,-1)_2(-\gamma ,-1)_2
(u_1\gamma -1,-1)_2  \\
&=&\Leg{-1}{\gamma }\Leg2d.
\end{eqnarray*}
For $p\mid a$, we put $P_p=(t,u_1,u_2,u_3) $ with
\[t=1, u_1=0, u_2^2=c^2{d'}, u_3=a\sqrt n.\] Observe that
\begin{eqnarray*}
&&(d'\beta  b-2\gamma u_2)(d'\beta b+2\gamma u_2)
={d'}^2\beta^2b^2-4\gamma^2 c^2d'  \\
&\equiv&2d'c^2(d'\beta^2-2\gamma^2)\equiv -2c^2n\alpha^2\pmod p.
\end{eqnarray*}
If $p\mid \alpha$, we choose $u_2$ such that $p\mid d'\beta  b+2\gamma u_2$; in addition, $p\nmid d'\beta  b-2\gamma u_2$, otherwise $p\mid\beta$ which contradicts that $(\alpha,\beta,\gamma)$ is a positive primitive integer solution to (\ref{root-2d-rank k-1 case}). If $p\nmid\alpha$, we have $p\nmid d'\beta  b\pm2\gamma u_2$. So we can always choose $u_2$ such that $p\nmid d'\beta  b-2\gamma u_2$. Since $p\mid a$, we get
\begin{eqnarray*}
\brBig{L_1L_3(P_p),-1}_p&=&\brBig{-u_2( d'\beta  b-2\gamma u_2+\alpha u_3),-1}_p  \\
&=&\brbig{ d'\beta b-2\gamma u_2,-1}_p=1.
\end{eqnarray*}
Similarly, for $p\mid b$, we have
\[\brBig{L_1L_3(P_p),-1}_p=1.\]
In summary, we have
\[\Ba\Lambda{\Lambda'}=\Leg{-1}{\gamma}\Leg2d.\]

Like the case $\rank A=k-2$, (1) is equivalent to $h_4(n)=1$ and $\Ba\Lambda{\Lambda'}=-1$, and Gauss genus theory implies that $\Leg{-1}\gamma=1$ if and only if $h_8(n)=1$ provided that $h_4(n)=1$. Therefore, (1) is equivalent to $h_4(n)=1$ and $h_8(n)\equiv \ALeg2d\equiv \frac{d-1}4\pmod2.$
This completes the proof of the theorem.
\qed\\

\section{Independence Property of Residue Symbols}
In this section, we assume that $a, b$ and $c $ are coprime positive integers
satisfying $a^2+b^2=2c^2$. Let $q_1, \cdots,q_{k'}$ be all the prime divisors of $abc$.
% For any $1\le l\le k'$, we denote by $\cR_l$ the set consisting of non-zero
% quadratic residue classes modulo $q_l$.

We first introduce some notation.
% Let $\beta=(\beta_1,\cdots,\beta_{k'})$ with all $\beta_l$ lying in $\cR_l$.
Given $k\ge2$,   let $\alpha=(\alpha_1,\cdots,\alpha_k)$ with all $\alpha_j\in\Brbig{1,5,9,13}$ and $\prod_{j=1}^k\alpha_j\equiv1\pmod8$.
Assume that $B=(B_{lj})_{k\times k} $ is a symmetric $\bF_2$ matrix with rank $k-2$ and $Bz_0=0$. Here $z_0=(1,\cdots,1)^{\rm T}\in\bF_2^k$. So there is a unique $z=(z_1,\cdots,z_k)^{\rm T}\not=0,z_0$ such that $Bz=0$ and $z_1=1$.
Our goal in this section is to estimate the number of $C_{k}(x,\alpha,B)$. Here $C_k(x,\alpha,B)$ consists of all $n=p_1\cdots p_k\le x$ satisfying
\begin{itemize}
\item $p_1<\cdots<p_k$ and $p_l\equiv\alpha_l\pmod{16}$ for all $1\le l\le k$,
\item $\ALeg{p_j}{p_l}=B_{lj}$ for all $1\le l<j\le k$,
\item $\Leg{p_l}{q_j}=1$ for all $1\le l\le k$ and $1\le j\le k'$, and
\item $\Leg{d'}d_4\Leg d{d'}_4=-1$ with $d=\prod_{l=1}^k p_l^{z_l}$ and $d'=\prod_{l=1}^kp_l^{1-z_l}$.
\end{itemize}

Due to the existence of the quartic residue symbols, we can't estimate $\#C_k(x,\alpha,B)$ directly. This problem can be solved by identifying $C_k(x,\alpha,B)$ with a set   counting corresponding integers over $\bZ[i]$. To this purpose, we first introduce some notation. Denote by $\cP$ the set of  primary primes of $\bZ[i]$ with positive imaginary part. We define $N$ to be the norm  map from $\bZ[i]$ to $\bZ$. Let $C_k'(x,\alpha,B)$ be all   $\eta=\lambda_1\cdots\lambda_k$ satisfying
\begin{itemize}
\item $N\eta\le x$ and $N\lambda_1<\cdots<N\lambda_k$,
\item $\lambda_l\in\cP$ and $N\lambda_l\equiv\alpha_l\pmod{16}$ for all $1\le l\le k$,
\item $\ALeg{N\lambda_j}{N\lambda_l}=B_{lj}$ for all $1\le l<j\le k$,
\item $\Leg{N\lambda_l}{q_j}=1$ for all $1\le l\le k$ and $1\le j\le k'$, and
\item $\Leg{\lambda'}{\lambda}=-1$ with $\lambda=\prod_{l=1}^k\lambda_l^{z_l}$ and
$\lambda'=\prod_{l=1}^k \lambda_l^{1-z_l}$.
\end{itemize}
\begin{lemma}\label{lem-identify-Ck(x,alpha,B)}
The following map is a bijection
\[C_k'(x,\alpha,B)\longrightarrow C_k(x,\alpha,B),\quad \eta\mapsto N\eta.\]
\end{lemma}
\begin{proof}
Let $\eta=\lambda_1\cdots\lambda_k$ be an element of $C_k'(x,\alpha,B)$ satisfying $N\lambda_1<\cdots<N\lambda_k$ and $\lambda_l\in\cP$ for all $1\le l\le k$. Denote by $p_l=N\lambda_l$ for all $1\le l\le k$. To show that $N\eta\in C_k(x,\alpha,B)$, we only need to verify $\Leg{d'}d_4\Leg d{d'}_4=-1$ with $d=N\lambda$ and $d'=N\lambda'$.
Note that
\begin{eqnarray*}
\Leg{p_l}{p_j}_4\Leg{p_j}{p_l}_4&=&
\Leg{\lambda_l\overline{\lambda_l}}{\lambda_j}_4
\Leg{\lambda_j\overline{\lambda_j}}{\lambda_l}_4  \\
&=&\Leg{\lambda_l}{\lambda_j}_4\Leg{\overline{\lambda_l}}{\lambda_j}_4
\Leg{\lambda_j}{\lambda_l}_4\Leg{\overline{\lambda_j}}{\lambda_l}_4 \\
&=& \Leg{\lambda_j}{\lambda_l}_4\Leg{{\lambda_j}}{\overline{\lambda_l}}_4
\Leg{\lambda_j}{\lambda_l}_4\Leg{\overline{\lambda_j}}{\lambda_l}_4.
\end{eqnarray*}
Here we have used the quadratic reciprocity law for $\Leg{\lambda_l}{\lambda_j}_4$ and   $\Leg{\overline{\lambda_l}}{\lambda_j}_4$.
From the definition of the quartic residue symbol, we have
$\Leg{{\lambda_j}}{\overline{\lambda_l}}_4\Leg{\overline{\lambda_j}}{\lambda_l}_4=1$.
Consequently,\[\Leg{p_l}{p_j}_4\Leg{p_j}{p_l}_4=
\Leg{\lambda_j}{\lambda_l}.\]
Therefore,
\[\Leg{d'}d_4\Leg d{d'}_4=\Leg{\lambda'}{\lambda}=-1.\]
So $N\eta$ lies in $C_k(x,\alpha,B)$.
The map is obviously injective. The map is surjective by observing  that for every rational prime $p\equiv1\pmod4$ there is exactly one primary prime in $\cP$ with norm $p$.
This completes the proof of the lemma.
\end{proof}

To use the idea of Cremona-Odoni \cite{cremona1989some}, we introduce another set $T(x)$. Here $T(x)$ is the set of positive integers $n=p_1\cdots p_{k-1}\le x$ satisfying
\begin{itemize}
\item $p_1<\cdots <p_{k-1}$,
\item $p_l\equiv \alpha_l\pmod{16}$ for all $1\le l\le k-1$,
\item $\ALeg{p_j}{p_l}=B_{lj}$ for all $1\le l<j\le k-1$, and
\item $\Leg{p_l}{q_j}=1$ for all $1\le l\le k-1$ and $1\le j\le k'$.
\end{itemize}
The independence property of Legendre symbols of Rhoades \cite{rhoades20092} implies
\begin{equation}\label{number-T(x)}
\#T(x)\sim 2^{-(k'+3)(k-1)-\binom{k-1}2}\cdot\#C_{k-1}(x),
\end{equation}
where $\binom k2$ is the binomial coefficient and $C_k(x)$ is the set of all positive square-free integers $n\le x$ with exactly $k$ prime factors.
Like $C_k(x,\alpha,B)$, we have to identify $T(x)$ with another set $T'(x)$. Here $T'(x)$ is the set of $\eta=\lambda_1\cdots\lambda_{k-1}$ satisfying
\begin{itemize}
\item $N\eta\le x$ and $N\lambda_1<\cdots<N\lambda_{k-1}$,
\item $\lambda_l\in\cP$ and $N\lambda_l\equiv\alpha_l\pmod{16}$ for all $1\le l\le k-1$,
\item $\ALeg{N\lambda_j}{N\lambda_l}=B_{lj}$ for all $1\le l<j<k$, and
\item $\Leg{N\lambda_l}{q_j}=1$ for all $1\le l<k$ and $1\le j\le k'$.
\end{itemize}
Similarly, we have the following lemma.
\begin{lemma}\label{lem-identify-T(x)}
The following map is a bijection
\[T'(x)\longrightarrow T(x),\quad \eta\mapsto N\eta.\]
\end{lemma}

Now we can use the idea of Cremona-Odoni \cite{cremona1989some} to prove the independence property of residue symbols.
\begin{thm}\label{thm-independence}
For $k$   an integer greater than $1$, let $\alpha=(\alpha_1,\cdots,\alpha_k)$ with $\prod_{j=1}^k\alpha_j\equiv1\pmod8$  and all $\alpha_j\in\Brbig{1,5,9,13}$.
Assume that $B=(B_{lj})_{k\times k} $ is a symmetric $\bF_2$ matrix with rank $k-2$ and the sum of its any given row being zero. Then
\[\#C_k(x,\alpha,B)=\frac{1+o(1)}{ 2^{3k+k'k+1+\binom k2}}\cdot\#C_k(x).\]
\end{thm}
\begin{proof}
Like Cremona-Odoni \cite{cremona1989some}, we consider the comparison map
\[f: C_k'(x,\alpha,B)\longrightarrow T'(x),\quad \eta\mapsto \eta/\tilde\eta,\]
where $\tilde\eta$ lies in $\cP$ such that its norm is the maximal prime divisor of $N\eta$. According to $z_k=0$ or not, the proof can be divided into two cases.

We first consider the case $z_k=0$. The next step is to consider the fiber of the comparison map $f$. Let $\epsilon=\prod_{j=1}^{k-1}\lambda_j\in T'(x)$ with $N\lambda_1<\cdots<N\lambda_{k-1}$ and all $\lambda_j\in\cP$. Then $\epsilon$ lies in $\Im f$ if and only if there exists a  $\lambda_0\in\cP$ with $N\lambda_0$ lying in $(N\lambda_{k-1}, x/N\epsilon]$ such that
\begin{enumerate}
\item[(1)] $N\lambda_0\equiv\alpha_k\pmod{16}$ and $\ALeg{N\lambda_0}{N\lambda_l}=B_{lk}$ for all $1\le l\le k-1$,
\item[(2)] $\Leg{N\lambda_0}{q_j}=1$ for all $1\le j\le k'$ and
\item[(3)] $\Leg{\lambda_0}\lambda=-\Leg{\epsilon/\lambda}\lambda$ with $\lambda=\prod_{l=1}^{k-1}\lambda_l^{z_l}$.
\end{enumerate}
Then  there is a unique subset $\sA=\sA_\epsilon$ of $\brBig{\bZ[i]/16\epsilon \epsilon_1\epsilon_2\bZ[i]}^\times$ such that for any prime $\theta$ the integer $\theta\epsilon$ lies in $C_k'(x,\alpha,B)$ if and only if $\theta\in\cP$ and $\theta\in\sA$ satisfying $N\theta\in (N\lambda_{k-1},N\epsilon]$. Here  $\epsilon_1$ is the product of all   primary primes lying in $\cP$ and lying above $p$ with $p\mid abc$ and $p\equiv1\pmod4$, and  $\epsilon_2$ is the product of all prime factors $p$ of $abc$ with $ p\equiv3\pmod4$. The cardinality of $\sA$ is evaluated in the following lemma.
\begin{lemma}\label{lem-cardi-Aepsilon}
Let $\varphi(\epsilon)$ be the cardinality of $G=\brBig{\bZ[i]/\fc}^\times$ with  $\fc=\fc_\epsilon$ the ideal $16\epsilon\epsilon_1\epsilon_2\bZ[i]$.  Then
\[\#\sA=2^{-k-k'-4}\varphi(\epsilon).\]
\end{lemma}
\emph{Proof of Lemma \ref{lem-cardi-Aepsilon}.}
From the definition of $\sA$, we see that $\sA$ represents those primary classes $\beta$ of $G$
%=\brBig{\bZ[i]/16\epsilon \epsilon_1\epsilon_2\bZ[i]}^\times$
such that
\begin{enumerate}
\item[(a)] $N\beta\equiv\alpha_k\pmod{16}$ and $\ALeg{N\beta}{N\lambda_l}=B_{lk}$ for all $1\le l\le k-1$,
\item[(b)] $\Leg{N\beta}{q_j}=1$ for all $1\le j\le k'$, and
\item[(c)] $\Leg\beta\lambda=\Leg{\epsilon/\lambda}\lambda$ with $\lambda=\prod_{l=1}^{k-1}\lambda_l^{z_l}$.
\end{enumerate}
Via Chinese Remainder Theorem we obtain the following isomorphism
\[G\simeq \brBig{\bZ[i]/16\bZ[i]}^\times\times
\brbigg{\prod_{l=1}^{k-1}\brBig{\bZ[i]/\lambda_l\bZ[i]}^\times}\times
\brBig{\bZ[i]/\epsilon_1\bZ[i]}^\times\times
\brBig{\bZ[i]/\epsilon_2\bZ[i]}^\times,\]
where the map is given by $\beta\mapsto (\beta_0,\cdots,\beta_{k-1},\beta_1',\beta_2')$. Here $\beta_l$ denotes the class $\beta\pmod {\lambda_l\bZ[i]}$ if $1\le l\le k-1$ and $\beta\pmod{16\bZ[i]}$ if $l=0$, and $\beta_j'$ denotes the class $\beta\pmod{\epsilon_j\bZ[i]}$ for $j=1,2$. Note that the residue symbol $\Leg\cdot\lambda$ is only non-trivial on those $\brBig{\bZ[i]/\lambda_l\bZ[i]}^\times$-component with $z_l=1$. For any $1\le l\le k-1$, the condition $\ALeg{N\beta_l}{N\lambda_l}=B_{lk}$ takes up a half of the $\brBig{\bZ[i]/\lambda_l\bZ[i]}^\times$-component. Here we have used the isomorphism $\brBig{\bZ[i]/\lambda_l\bZ[i]}^\times\simeq \brBig{\bZ/N\lambda_l \bZ}^\times$ by $\lambda_l\in\cP$. The condition $\Leg\beta\lambda=-\Leg{\epsilon/\lambda}\lambda$ selects another half of the $\prod_{l=1}^{k-1}\brBig{\bZ[i]/\lambda_l\bZ[i]}^\times$-component.
Similarly, we have the following isomorphisms
\[\brBig{\bZ[i]/\epsilon_1\bZ[i]}^\times\simeq \prod_{\lambda_j'\mid \epsilon_1}
\brBig{\bZ[i]/\lambda_j'\bZ[i]}^\times\quad{\rm and}\quad
\brBig{\bZ[i]/\epsilon_2\bZ[i]}^\times\simeq \prod_{q_j\mid \epsilon_2} \brBig{
\bZ[i]/q_j\bZ[i]}^\times.\]
Like above, $\Leg{N\beta_1'}{q_j}=1$ selects a half of the $\brBig{\bZ[i]/\lambda_j'\bZ[i]}^\times$-component provided that $\lambda_j'\mid\epsilon_1$. For $q_j\mid \epsilon_2$, the condition $\Leg{N\beta_2'}{q_j}=1$ also selects a half of the $\brBig{\bZ[i]/q_j\bZ[i]}^\times$-component by the composition of the homomorphisms
\[\brBig{\bZ[i]/q_j\bZ[i]}^\times \twoheadrightarrow \brBig{\bZ/q_j\bZ}^\times \rightarrow \Brbig{\pm1}. \]
Here the last homomorphism is given by the Legendre symbol.

Now we consider the $G'$-component, where $G'=\brBig{\bZ[i]/16\bZ[i]}^\times$. The primary condition selects $G_0=1+(2+2i)\bZ[i]$ of $G'$ and $\#G'=4\#G_0$. Then $N\beta\equiv\alpha_k\pmod{16}$ selects a fourth of $G_0$. So
\[\#\sA=2^{-k-k'-4}\varphi(\epsilon).\]
This completes the proof of the lemma.
\qed\\

For any $\epsilon\in T'(x)$,   let $h(\epsilon)$ be the number of primes $\theta\in\cP$ such that $\theta+16\epsilon\epsilon_1\epsilon_2\bZ[i]$ lies in $\sA$ and $N\theta$ lies in $(N\lambda_{k-1},N\epsilon]$. Then we have
\begin{equation}\label{num-C_k'(x,a,B)}
\#C_k'(x,\alpha,B)=\sum_{\epsilon\in T'(x)} h(\epsilon).
\end{equation}

We will divide the sum in (\ref{num-C_k'(x,a,B)}) into several parts according to the norm of $\eta$. To this purpose, we introduce some notation. We define $\mu$ and $\nu$ to be $(\log x)^{100}$ and $\exp\brBig{\frac{\log x}{(\log\log x)^{100}}}$ respectively. For a set $M$ consisting of positive integers and a function $g$ on $\bZ[i]$, we define $$\oset*{\sum_{N\delta\in M}}g(\delta)=\sum_{\delta\in T'(x)\atop N\delta\in M} g(\delta).$$
Like Lemma 3.1 of Cremona-Odoni, we have the following lemma (the proof is the same as that of Cremona-Odoni).
\begin{lemma}\label{lem-Cremona}
If $m=20$ and $n=\mu$, then
\[\oset*{\sum_{m<N\delta\le n}}\Li(x/N\delta)=\oBig{  \frac{x(\log\log x)^{k-1}}{\log x} }.\]
Similar estimation is true for $m=\nu$ and $n=x^{\frac{k-1}k}$. Moreover, we have
\[\oset*{\sum_{\mu<N\delta\le \nu}}\Li(x/N\delta)\sim \frac{\#T'(x)}{k-1}\log\log x.\]
\end{lemma}

Now we divide the interval $(1,x]$ into five parts by the points $20, \mu, \nu$ and $x^{\frac{k-1}k}$. First, for those $\epsilon\in T'(x)$ satisfying $N\epsilon\le 20$, we get $h(\epsilon)\le \pi(x/N\epsilon)$ by noting that every prime ideal corresponds to at most one primary prime element. Here $\pi(y)$ denotes the number of those prime ideals with norm no larger than $y$, and  the prime ideal theorem says that \[\pi(y)\sim \Li(y).\]So we have
\[\oset*{\sum_{N\epsilon\le20}}h(\epsilon)=\Obig{\Li(x)}.\]
Next, if $\epsilon\in T'(x)$ and $20<N\epsilon\le \mu$, then we also have $h(\epsilon)=\Obig{\Li(x/N\epsilon)}$. Thus Lemma \ref{lem-Cremona} implies that
\[\oset*{\sum_{20<N\epsilon\le \mu}}h(\epsilon)=\oBig
{  \frac{x(\log\log x)^{k-1}}{\log x} }.\]
By the same reason, we obtain
\[\oset*{\sum_{\nu<N\epsilon\le x^{\frac{k-1}k}}} h(\epsilon)=\oBig
{  \frac{x(\log\log x)^{k-1}}{\log x} }.\]
Next, if $N\epsilon>x^{\frac{k-1}k}$, then $N\lambda_{k-1}>x^{\frac1k}$. So
$x/N\epsilon<x^{\frac1k}<N\lambda_{k-1}<N\theta$. Thus from the definition of $h(\epsilon)$ we get $h(\epsilon)=0$ in this case. Therefore,
\[\oset*{\sum_{x^{\frac{k-1}k}<N\epsilon\le x}}h(\epsilon)=0.\]
From these estimations, we see that
\begin{equation*}
\#C_k'(x,\alpha,B)\sim\frac12\oset*{\sum_{\mu<N\epsilon\le\nu}}
\pi'(x/N\epsilon,\sA,16\epsilon\epsilon_1\epsilon_2)-
\half\oset*{\sum_{\mu<N\epsilon\le \nu}}
\pi'(N\lambda_{k-1},\sA,16\epsilon\epsilon_1\epsilon_2).
\end{equation*}
Here the factor $\half$ comes from $\theta\in\cP$, and $\pi'(y,\sB,\gamma)$ denotes the number of prime elements $\theta$ of $\bZ[i]$ such that $N\theta\le y$ and $\theta+\gamma\bZ[i]$ lies in $\sB$. Noting that
\[\oset*{\sum_{\mu<N\epsilon\le \nu}}
\pi'(N\lambda_{k-1},\sA,16\epsilon\epsilon_1\epsilon_2)
\ll \nu \Li(\nu)=\oBig
{  \frac{x(\log\log x)^{k-1}}{\log x} },\] we derive
\begin{equation}\label{num-C_k'(x,a,B)-mu to nu}
\#C_k'(x,\alpha,B)\sim\frac12\oset*{\sum_{\mu<N\epsilon\le\nu}}
\pi'(x/N\epsilon,\sA,16\epsilon\epsilon_1\epsilon_2).
\end{equation}

To estimate the sum in the right hand side of (\ref{num-C_k'(x,a,B)-mu to nu}), we have to use the Dirichlet prime ideal theorem. This theorem only gives the distribution of prime ideals, while (\ref{num-C_k'(x,a,B)-mu to nu}) requires the estimation on the number of prime elements. We can solve this problem by the following transformation.
%  We first introduce some notation.
Via Theorem 6.1 of Lang \cite{lang110gtm}, we yield the exact sequence
\begin{eqnarray}\label{exact-seq}
\xymatrix{
1\ar[r]^{} &\bZ[i]^\times  \ar[r]^{} & \brbig{\bZ[i]/\fc}^{\times} \ar[r]^{\Phi} &
I(\fc)/P(\fc)  \ar[r]^{} & 1   }
\end{eqnarray}with $\fc=\fc_\epsilon$ defined in Lemma \ref{lem-cardi-Aepsilon}.
Here $\Phi$ is the map induced from $\Phi'$ which sends every $\fc$-invertible element $\beta$ of $\bZ[i]$ into the principal ideal $(\beta)$. For an ideal $\fa$ and a subset $\sS$ of $I(\fa)/P_\fa$, we define $\pi(y,\sS,\fa)$ to be the number of prime ideals $\fp$ such that $N\fp\le y$ and $\fp P_\fa$ lies in $\sS$. Let $\sT=\sT_\epsilon=\Phi(\sA)$. Then we get
\begin{equation}\label{num-transform}
\pi'(y,\sA,16\epsilon\epsilon_1\epsilon_2)=\pi(y,\sT,\fc),
\quad {\rm  }\quad \#\sA=\#\sT
\end{equation}
by noting that every prime ideal  in a class of $\sT$ corresponds to exactly one primary prime element.
From (\ref{num-C_k'(x,a,B)-mu to nu}) and (\ref{num-transform}) we get
\begin{equation}\label{num-last-version-C_k'(x,a,B)}
2\#C_k'(x,\alpha,B)\sim  \oset*{\sum_{\mu<N\epsilon\le \nu}} \pi(x/N\epsilon,\sT_\epsilon,\fc_\epsilon).
\end{equation}
By the   relation of $\psi(y,\sT,\fc)$ and $\pi(y,\sT,\fc)$, it suffices to estimate
\begin{equation}\label{sum-psi(x,sT,fc)}
\oset*{\sum_{\mu<N\epsilon\le \nu}} \psi(x/N\epsilon,\sT_\epsilon,\fc_\epsilon).
\end{equation}
Here $\psi(y,\sT,\fc)$ is given by
\[\sum_{N\fa\le y\atop \fa P_\fc\in\sT} \Lambda(\fa).\]
Via the orthogonality of characters and the exact sequence (\ref{exact-seq}), we get
\[\psi(y,\sT_\epsilon,\fc_\epsilon)=\frac4{\varphi(\epsilon)}\sum_{\chi \mod \fc_\epsilon} \psi(y,\chi)\sum_{[\fa]\in\sT_\epsilon} \overline{\chi(\fa)}.\]
Here $\chi$ runs over all characters of $I(\fc_\epsilon)/P_{\fc_\epsilon}$ and
\[\psi(y,\chi)=\sum_{N\fa\le y} \Lambda(\fa)\chi(\fa).\]
Applying this formula to (\ref{sum-psi(x,sT,fc)}), we derive
\[\oset*{\sum_{\mu<N\epsilon\le \nu}} \psi(x/N\epsilon,\sT_\epsilon,\fc_\epsilon)
=\oset*{\sum_{\mu<N\epsilon\le \nu}} \frac4{\varphi(\epsilon)}\sum_{\chi \mod \fc_\epsilon} \psi(x/N\epsilon,\chi)\sum_{[\fa]\in\sT_\epsilon} \overline{\chi(\fa)}.\]
Separating out all the principal characters, we get
\begin{equation}\label{sum-I and II}
\oset*{\sum_{\mu<N\epsilon\le \nu}} \psi(x/N\epsilon,\sT_\epsilon,\fc_\epsilon)
=(I)+(II)
\end{equation}with
\begin{eqnarray*}
(I)&=&\oset*{\sum_{\mu<N\epsilon\le \nu}}\frac{4\#\sT_\epsilon}{\varphi(\epsilon)} \psi(x/N\epsilon,\chi_0),  \\
(II)&=&\oset*{\sum_{\mu<N\epsilon\le \nu}} \frac4{\varphi(\epsilon)}\sideset{}'\sum_{\chi \mod \fc_\epsilon} \psi(x/N\epsilon,\chi)\sum_{[\fa]\in\sT_\epsilon} \overline{\chi(\fa)}.
\end{eqnarray*}Here $\sideset{}'\sum$ is the sum over all non-principal characters of a fixed modulus.

We first treat the main term (I). Via Lemma \ref{lem-cardi-Aepsilon} we obtain that
\begin{eqnarray*}
(I)&=&2^{-k-k'-2}    \oset*{\sum_{\mu<N\epsilon\le \nu}}    \psi(x/N\epsilon)  \\
&\sim&  2^{-k-k'-2}  \oset*{\sum_{\mu<N\epsilon\le \nu}}
\log(x/N\epsilon)\Li(x/N\epsilon)  \\
&\sim &  2^{-k-k'-2}\log x  \oset*{\sum_{\mu<N\epsilon\le \nu}}
\Li(x/N\epsilon)  \\
&\sim&\frac1{(k-1)\cdot2^{k+k'+2}}\#T'(x)\cdot\log x\cdot\log\log x.
\end{eqnarray*}

Now we assume that $(II)$ is error term to prove the theorem. By (\ref{num-last-version-C_k'(x,a,B)}) and (\ref{sum-psi(x,sT,fc)}), we have
\[\#C_k'(x,\alpha,B)\sim \frac1{(k-1)\cdot 2^{k+k'+3}}\#T'(x)\cdot\log\log x
\sim \frac1{2^{k'k+3k+1+\binom k2}}\#C_k(x).\]
Here we have used (\ref{number-T(x)}) and Lemma \ref{lem-identify-T(x)}. So by lemma \ref{lem-identify-Ck(x,alpha,B)} the proof of the theorem is reduced to showing that $(II)$ is an error term.

Like Cremona-Odoni, we have to separate out all the modulus in the sum $(II)$ with possible Siegel zeros.   We denote by $\dag_1$  the conductor of the exceptional primitive conductor with $Z\le 256\nu$ in Page Theorem (Proposition \ref{thm-Page}). According to   the modulus $\dag$ being a multiple of $\dag_1$ or not, we can divide the sum $(II)$ into the subsums $(III)$ and $(IV)$, where
\begin{eqnarray*}
(III)&=&\oset*{\sum_{\mu<N\epsilon\le \nu\atop \dag_1\mid\fc_\epsilon}} \frac4{\varphi(\epsilon)}\sideset{}'\sum_{\chi \mod \fc_\epsilon} \psi(x/N\epsilon,\chi)\sum_{[\fa]\in\sT_\epsilon} \overline{\chi(\fa)},  \\
(IV)&=&\oset*{\sum_{\mu<N\epsilon\le \nu\atop \dag_1\nmid\fc_\epsilon}} \frac4{\varphi(\epsilon)}\sideset{}'\sum_{\chi \mod \fc_\epsilon} \psi(x/N\epsilon,\chi)\sum_{[\fa]\in\sT_\epsilon} \overline{\chi(\fa)}.
\end{eqnarray*}
We use the trivial estimation $\psi(x/N\epsilon,\chi)\ll \psi(x/N\epsilon)$ to bound $(III)$ and get
\begin{eqnarray*}
(III)&\ll&\oset*{\sum_{\mu<N\epsilon\le \nu\atop \dag_1\mid\fc_\epsilon}}\psi(x/N\epsilon)\ll x\oset*{\sum_{\mu<N\epsilon\le \nu\atop \dag_1\mid\fc_\epsilon}}(N\epsilon)^{-1}  \\
&=&\frac x{N\dag_1}\sum_{\mu<t N\dag_1 \le \nu}t^{-1}\sum_{\epsilon\in T'(\infty)\atop \dag_1\mid \fc_\epsilon, N\epsilon=tN\dag_1}1\ll  \frac{x\log \nu}{N\dag_1}.
\end{eqnarray*}
Again by Page Theorem (Proposition \ref{thm-Page}) with $Z=256\nu$, there is a positive constant $c_6$ such that the Siegel zero $\beta$ of the primitive character  with modulus $\dag_1$ has the property
\[\beta>1-\frac{c_6}{\log 256\nu}.\]
Via Siegel Theorem (Proposition \ref{thm-Siegel}), for any $\epsilon'>0$ (not confused with our $\epsilon\in T'(x)$), there is a constant $c'=c'(\epsilon',2)>0$ such that
\[\beta\le 1-c'(4N\dag_1)^{-\epsilon'}.\]
Taking $\epsilon'=1/200$, we obtain $N\dag_1\gg_{\epsilon'} (\log\nu)^{100}$. Consequently,\[(III)\ll x(\log\nu)^{-99}.\]

Next, we bound the sum $(IV)$. There is no Siegel zero in $(IV)$. So we can apply the explicit formula  (\ref{eq:explicit-we need}) with $T=(N\epsilon)^4$ to all the $\psi(x/N\epsilon,\chi)$ in $(IV)$ and obtain
\[\psi(x/N\epsilon,\chi)\ll x(N\epsilon)^{-1}(\log x)^2 \exp
\brBig{-\frac{c_7\log(x/N\epsilon)}{\log N\epsilon}}+x(N\epsilon)^{-5}(\log x)^2+
x^{\frac14}(N\epsilon)^{-\frac14} \log(x/N\epsilon).\] Correspondingly,   $(IV)$ is bounded (up to a constant) by the sum of the three  sums
\begin{eqnarray*}
(V)&=&\oset*{\sum_{\mu<N\epsilon\le \nu\atop \dag_1\nmid\fc_\epsilon}}
x(N\epsilon)^{-1}(\log x)^2 \exp
\brBig{-\frac{c_7\log(x/N\epsilon)}{\log N\epsilon}}, \\
(VI)&=&\oset*{\sum_{\mu<N\epsilon\le \nu\atop \dag_1\nmid\fc_\epsilon}}
x(N\epsilon)^{-5}(\log x)^2,  \\
(VII)&=&\oset*{\sum_{\mu<N\epsilon\le \nu\atop \dag_1\nmid\fc_\epsilon}}
x^{\frac14}(N\epsilon)^{-\frac14} \log(x/N\epsilon).
\end{eqnarray*}
For the sum $(V)$, we have
\begin{eqnarray*}
(V)&\ll&x(\log x)^2\exp\brBig{-c_8(\log\log x)^{100}}\cdot
\oset*{\sum_{\mu<N\epsilon\le \nu\atop \dag_1\nmid\fc_\epsilon}} (N\epsilon)^{-1}  \\
&\ll&x(\log x)^3\exp\brBig{-c_8(\log\log x)^{100}}.
\end{eqnarray*}
Similarly, we get
\begin{eqnarray*}
(VI)&\ll&x(\log x)^2\mu^{-3}\ll x(\log x)^{-200},  \\
(VII)&\ll&x^{\frac14}\log x \cdot\nu^{\frac34}\ll x^\half.
\end{eqnarray*}
Consequently, the sum $(IV)$ is an error term. This completes the proof of the theorem.
\end{proof}

\section{Distribution Result}

In this section, we assume that $a, b$ and $c $ are coprime positive integers
such that $a^2+b^2=2c^2$ and the  dimension of the $2$-Selmer group of $E$  is
two.  Let $q_1, \cdots,q_{k'}$ be all the prime divisors of $abc$. Denote by $k$   a
fixed positive integer.
%  We denote by $k'$ the number of different prime factors of $abc$.

Let $n=p_1\cdots p_k\in \sQ_k(x)$ with $p_1<\cdots<p_k$. By Theorem \ref{thm-ab odd square-p 1 mod 4-introd}, $n$ lies in $\sP_k(x)$ if and only if $h_4(n)=1$ and $h_8(n)\equiv\frac{d-1}4\pmod2$. Here $d$ is a certain divisor of $n$. From the proof of Theorem \ref{thm-ab odd square-p 1 mod 4-introd}, the characterization of $n\in\sP_k(x)$ is divided into two cases (according to the rank of $A=A_n$ being $k-1$ or $k-2$).

We first assume that $\rank A=k-2$. Then $h_4(n)=1$ is equivalent to $\fb\not\in\Im A$, where $\fb=\brBig{\ALeg2{p_1},\cdots,\ALeg2{p_k}}^{\rm T}$. Since $\rank A=k-2$ and $Az_0=0$ with $z_0=(1,\cdots,1)^{\rm T}\in\bF_2^k$,  there is a unique column vector $z=(z_1,\cdots,z_k)^{\rm T}\not=0, z_0$  with $z_1=1$ such that $Az=0$. Then $d=\prod_{l=1}^k p_l^{z_l}$ is congruent to $5$ modulo $8$. Jung-Yue \cite{yueJung2011eightrank} (Theorem 3.3 (ii)) showed that in this case $h_8(n)=1$ is equivalent to
\begin{equation}\label{dd'-4residue symbol}
\Leg d{d'}_4\Leg{d'}d_4=-1,
\end{equation}
where $d'=\prod_{l=1}^kp_l^{1-z_l}$.

Now we assume that $\rank A=k-1$. Then $h_4(n)=1$ and $\fb\in\Im A$ with $\fb=\brBig{\ALeg2{p_1},\cdots,\ALeg2{p_k}}^{\rm T}$.  Let $z=(z_1,\cdots,z_k)^{\rm T}$ be a column vector with $Az=\fb$. Then $d=\prod_{l=1}^{k}p_l^{z_l}$. Jung-Yue \cite{yueJung2011eightrank} (Theorem 3.3 (iii) and (iv)) proved that $h_8(n)=1$ if and only if $\Leg{2d}{d'}_4\Leg{2d'}d_4=(-1)^{\frac{n-1}8}$. Here $d'=\prod_{l=1}^kp_l^{1-z_l}$.\\

Now we  begin to prove Theorem \ref{thm-dist}.\\
\emph{Proof of Theorem \ref{thm-dist}.} Like Theorem \ref{thm-ab odd square-p 1 mod 4-introd}, we also divide the proof into two cases.

First, we count the number $N_1(x)$ of those $n\in\sP_k(x)$ with $\rank A_n=k-2$. Obviously, we have $k\ge2$ in this case. We first introduce some notation. Let $\sB$ be the set of $k\times k$ symmetric matrices $B$ over $\bF_2$ with $\rank B=k-2$ and $Bz_0=0$, where $z_0=(1,\cdots,1)^{\rm T}\in \bF_2^k$. We define $\sI$ to be all the $\alpha=(\alpha_1,\cdots,\alpha_k)$ with $\prod_{l=1}^k \alpha_l\equiv1\pmod8$ and   $\alpha_l\in\Brlr{ 1, 5, 9, 13 }$ for $1\le l\le k$.
Given $B\in\sB$, we denote by $\sI_B$ the set of $\alpha\in\sI$ such that $\fb_\alpha$ does not  lie in the image of $B$. Here $\fb_\alpha=\brBig{\ALeg{2}{\alpha_1},\cdots,\ALeg2{\alpha_k}}^{\rm T}$. Note that $\fb_\alpha^{\rm T} z_0=0$ by $\prod_{l=1}^k \alpha_l\equiv1\pmod8$.
For any $B\in \sB$ and $\alpha\in\sI_B$,  by (\ref{dd'-4residue symbol}) those $n=p_1\cdots p_k\in \sP_k(x)$ satisfying
\begin{itemize}
\item $p_1<\cdots<p_k$ and $A_n=B$,
\item $\Leg{p_l}{q_j}=1$ for all $1\le l\le k$ and $1\le j\le k'$, and
\item $p_l\equiv\alpha_l\pmod{16}$ for all $1\le l\le k$
\end{itemize}
consist the set $C_k(x,\alpha,B)$. Moreover, given $B\in\sB$ and $\alpha\in\sI-\sI_B$,  the intersection of $C_k(x,\alpha,B)$  and $\sP_k(x)$ is empty. Therefore, the number $N_1(x)$ of those $n\in\sP_k(x)$ with $\rank A_n=k-2$ is
\begin{equation}\label{num-N1(x)}
N_1(x)=\sum_{B\in\sB}\sum_{\alpha\in\sI_B}\#C_k(x,\alpha,B)\sim 2^{-3k-kk'-1-\binom k2}\cdot\#C_k(x)\cdot\sum_{B\in\sB}\#\sI_B.
\end{equation}
Here we have used Theorem \ref{thm-independence}.

Now we count the number of $\sI_B$ with $B$ given. First, given $\fb=(b_1,\cdots,b_k)^{\rm T}$ with $\fb\not\in\Im B$ and $\fb^{\rm T}z_0=0$, we  count the number of $\alpha$ such that $\fb=\fb_\alpha$. As $\fb=\fb_\alpha$, we get $\ALeg2{\alpha_l}=b_l$ for all $1\le l\le k$. So any $\alpha_l$ has exactly two choices. Thus the number of $\alpha$ such that $\fb=\fb_\alpha$ is $2^k$.
Next, we count the number of column vectors $\fb$ such that $\fb^{\rm T}z_0=0$ and $\rank \left(
         \begin{array}{cc}
           B & \fb \\
         \end{array}
       \right)=k-1
$. Since $  \left(
         \begin{array}{cc}
           B & \fb \\
         \end{array}
       \right)z_0=0$ and $B$ is symmetric, we get
$\rank B'=k-2$ and $\rank \left(
         \begin{array}{cc}
           B' & \fb' \\
         \end{array}
       \right)=k-1.$ Here $B'$ is the matrix obtained from $B$ by deleting its
last row and column, and $\fb'$ is the vector obtained from $\fb$ by deleting its last component. Thus $\fb'$ does not lie in the image of $B'$. So there are $2^{k-2} $ many such $\fb'$ and $\fb$. Consequently, $\#\sI_B=2^{2k-2}$. Then (\ref{num-N1(x)}) implies that\[N_1(x)\sim 2^{-k-k'k-3-\binom k2}\cdot\#C_k(x)\cdot\#\sB.\]
The number of $\sB$ can be obtained from the following result of Brown et al \cite{brown2006trivial}.
\begin{prop}
For positive integers $r\le k$, we denote by $\sB_{k,r}$ the set of $k\times k$ symmetric matrices over $\bF_2$ with rank $r$. Then
\[\#\sB_{k,r}=u_{r+1} 2^{\binom{r+1}2}\cdot
\prod_{l=0}^{k-r-1}\frac{2^k-2^l}{2^{k-r}-2^l}\]
with $u_l$ defined above Theorem \ref{thm-dist}.
\end{prop}
Note that the map sending $B$ to $B'$ induces a bijection between $\sB$ and $\sB_{k-1,k-2}$. So
\[\#\sB=u_{k-1}2^{\binom{k-1}2}\cdot(2^{k-1}-1).\]
We get
\[N_1(x)\sim 2^{-k-k'k-3}(1-2^{1-k})u_{k-1}\cdot\#C_k(x).\]

Finally, we counts the number $N_2(x)$ of $n\in\sP_k(x)$ with $\rank A_n=k-1$. Like above (we refer to our previous paper \cite{wang2016-3} for detailed proof), we get \[N_2(x)\sim 2^{-k-k'k-2}u_k\cdot\#C_k(x).\]
This finishes the proof of the theorem.
\qed\\

\textbf{Acknowledgements} 
This work is supported by the Fundamental Research Funds for the Central Universities (Grant No. GK201703004). The author is greatly indebted to his advisor Professor Ye Tian for many instructions and suggestions.

\noindent Zhangjie Wang, \\
School of Mathematics and Information Science, \\
Shaanxi Normal University, Xi'an 710119,China. \\
{\it {zhangjiewang@snnu.edu.cn}  }

%  ÖÐÑë¸ßУ¾­·Ñ±àºÅ GK201703004

%\bibliographystyle{alpha}  % alpha   %ieeetr
%\bibliography{00Ref}

\end{CJK}
\end{document}